\documentclass{amsart}

\pdfoutput=1

\usepackage{yhmath}

\newcommand{\C}{\mathrm{C}}

\newcommand{\R}{\mathbb{R}}

\newcommand{\N}{\mathbb{N}}

\usepackage{stmaryrd}
\usepackage{tikz}

\usepackage{amssymb}
\usepackage{microtype}
\usepackage{pgfplots}
\pgfplotsset{compat=1.18} 
\usetikzlibrary{decorations.pathreplacing}
\usepackage[pdftex,colorlinks,citecolor=black,linkcolor=black,urlcolor=black,bookmarks=false]{hyperref}
\def\MR#1{\href{http://www.ams.org/mathscinet-getitem?mr=#1}{MR#1}}

\usepackage{doi}
\usepackage{booktabs}
\usepackage[c2]{optidef}
\usepackage{ytableau}
\usepackage{algorithm}
\usepackage{algpseudocode}
\usepackage[tableposition=above]{caption}

\newtheorem{theorem}{Theorem}[section]
\newtheorem{proposition}[theorem]{Proposition}

\newtheorem{corollary}[theorem]{Corollary}
\newtheorem{lemma}[theorem]{Lemma}
\theoremstyle{remark}

\theoremstyle{definition}

\numberwithin{equation}{section}
\numberwithin{table}{section}
\numberwithin{figure}{section}
\title[Solving clustered low-rank semidefinite programs]{Solving clustered low-rank semidefinite programs arising from polynomial optimization}

\author{David de Laat}
\address{D.\ de Laat, Delft Institute of Applied Mathematics\\
Delft University of Technology\\ Delft, The Netherlands} \email{d.delaat@tudelft.nl}

\author{Nando Leijenhorst}
\address{N.\ M.\ Leijenhorst, Delft Institute of Applied Mathematics\\
Delft University of Technology\\ Delft, The Netherlands} \email{n.m.leijenhorst@tudelft.nl}

\date{August 21, 2024} 

\begin{document}

\maketitle

\begin{abstract}
We study a primal-dual interior point method specialized to clustered low-rank semidefinite programs requiring high precision numerics, which arise from certain multivariate polynomial (matrix) programs through sums-of-squares characterizations and sampling. We consider the interplay of sampling and symmetry reduction as well as a greedy method to obtain numerically good bases and sample points. We apply this to the computation of three-point bounds for the kissing number problem, for which we show a significant speedup. This allows for the computation of improved kissing number bounds in dimensions $11$ through $23$. The approach performs well for problems with bad numerical conditioning, which we show through new computations for the binary sphere packing problem.
\end{abstract}

\section{Introduction}

In discrete geometry, many of the best-known bounds on the optimal cardinality of spherical codes  \cite{delsarte_spherical_1977,MR2415397,bachoc_new_2007,mittelmann_high_2010,machado_improving_2018}, optimal sphere packing densities \cite{cohn_new_2003,de_laat_upper_2014}, optimal densities for packings with nonspherical shapes \cite{MR3679945}, and optimal ground state energies  \cite{MR1181534,MR2947943,de_laat_moment_2019} are obtained using linear programming and semidefinite programming bounds. Similar approaches are used in analytic number theory \cite{chirre_pair_2020} and the conformal bootstrap \cite{simmons-duffin_semidefinite_2015}. 

These bounds are derived using constraints on $k$-point correlations for small $k$, which leads to conic optimization problems involving positive semidefinite matrix variables as well as polynomial inequality constraints in a modest number of  variables. Such problems are then solved using semidefinite programming via sums-of-squares characterizations for the polynomial constraints. 

For applications in discrete geometry, we often need solutions of rather high precision, so we need to use \emph{second-order} interior point methods, which have linear convergence. Moreover, due to the orthogonal bases in the formulations of the problems, the problems have seemingly unavoidable bad numerical conditioning. In practice, computations are therefore performed using the general purpose semidefinite programming solvers \texttt{SDPA-QD} and \texttt{SDPA-GMP} \cite{yamashita_high-performance_2010}, which both use high-precision numerics, and computations regularly take weeks to complete (see, e.g., \cite{machado_improving_2018}).

When the constraint matrices defining a semidefinite program are of low rank, this can be exploited in the interior-point method 
\cite{lofberg_coefficients_2004,zhang_liu_low-rank_2007}, which has been used in the solvers \cite{benson_dsdp5_nodate,anjos_implementation_2012,simmons-duffin_semidefinite_2015}. As observed by Parrilo and L\"ofberg \cite{lofberg_coefficients_2004}  rank $1$ constraints naturally appear from sum-of-squares characterizations when sampling is used as opposed to coefficient matching. 

In \cite{simmons-duffin_semidefinite_2015,landry_scaling_2019} Simmons-Duffin develops a high-precision solver that uses this rank $1$ structure. The solver exploits clustering in the constraints, supports parallelization, and works very well in practice. The implementation however has been designed for problems in the conformal bootstrap and only supports a specific form of semidefinite programs that arise from univariate polynomial optimization problems. Inspired by this success, the goal of this paper is to explore the use of low-rank constraints for solving optimization problems such as those arising in discrete geometry.

That low-rank constraints can be exploited is not obvious, since problems in discrete geometry often have large symmetry groups, and symmetry reduction leads to constraints on the positive semidefinite matrix variables which may be of large rank and which are no longer sparse due to the sampling approach. Moreover, it also leads to multivariate polynomial inequality constraints which are invariant under certain group actions, and exploiting this leads to constraints matrices of rank greater than one.

Our first contribution is the implementation of a high-precision, primal-dual, interior-point solver that can exploit more general low-rank structures for the constraint matrices than only rank $1$ constraints.\footnote{See \href{https://github.com/nanleij/ClusteredLowRankSolver.jl}{github.com/nanleij/ClusteredLowRankSolver.jl} for  source code and documentation.} The solver is written in the high-level language Julia \cite{bezanson_julia_2017}, and is implemented in such a way that fast matrix-matrix multiplication can be exploited (using Arb \cite{johansson_arb_2017}). It comes with a user-friendly interface for modeling problems involving both low-rank semidefinite constraints, as well as low-rank polynomial constraints, and it can automatically convert between these. Similar to \cite{simmons-duffin_semidefinite_2015} the solver can exploit clustering of the constraints (where clusters of positive semidefinite matrix blocks are linked through free variables), and the solver has a custom parallelization approach tailored to problems where we have fewer clusters, but many samples per cluster. Note that many of the features of our implementation are already present in an existing solver, but none of the existing solvers implements all of them at the same time.

Secondly, we study the interplay of sampling and symmetry reduction, which has not been done before. We give necessary and sufficient conditions on the sample set in the presence of symmetry. We furthermore show empirically that a greedy approach to finding good samples, and transforming the bases to be orthogonal with respect to these samples, works well in this setting. We also investigate an approach where we remove dense constraint matrices by parameterizing them by free variables, which allows for more clustering; see Section~\ref{sec:kissing_number}.

Our third contribution is to emperically show the speed and stability of the approach described in the previous paragraphs by considering two applications from discrete geometry. First, we consider the three-point bound for the kissing number problem \cite{bachoc_new_2007}. 
We consider this problem because it involves invariant, multivariate polynomial inequality constraints and because it has additional positive semidefinite matrix variables for which the clustering  approach we use plays a role.  Moreover, it is an important problem in discrete geometry for which extensive computations have already been performed. We show a significant speedup compared to previous computations performed with \texttt{SDPA-GMP}, by a factor $28$ for the most extensive computations previously performed. This allows us to perform computations using polynomials of degree $40$ as opposed to degree $32$ (for which extrapolation shows this approach would have been faster by a factor $40$), which also results in improved kissing number bounds in dimension $11$-$23$.

Then we consider the adaptation of the Cohn-Elkies bound for the binary sphere packing problem \cite{de_laat_upper_2014}. We show this bound can be written using matrix polynomial inequality constraints, and we use this to perform new computations which highlight the numerical stability of the solver. These computations show the bounds are not necessarily convex, can beat Florian's bound in dimensions $2$, and may converge to the optimal density in dimensions $8$ and $24$ as the ratio of the radii goes to $0$.

These two representative examples show that exploiting low-rank constraints can be very beneficial for the $k$-point bounds arising in discrete geometry. We expect this will not just speed up existing computations, but will also allow for tackling more difficult problems which were previously out of reach. 

\section{A specialized interior point method}
\label{sec:algorithm}

In this section we give an exposition of the primal-dual interior-point method for semidefinite programming as used by \texttt{SDPB} \cite{simmons-duffin_semidefinite_2015}, which in turn builds on \texttt{SDPA} \cite{yamashita_high-performance_2010}. We generalize the method to a very general low-rank structure (see \eqref{eq:constraintmatrices} and \eqref{eq:nonsymmetric}), and we show how this can be exploited in the computation of the Schur complement matrix in a way that fast matrix-matrix multiplication can be employed (which is especially beneficial because we use high-precision arithmetic).  
Because our applications consist of problems in extremal geometry, which typically have few clusters and a large number of constraints within a cluster, our  parallelization strategy is different from  \texttt{SDPB}.
The interior point method uses the $XZ$ search direction \cite{MR1387330,MR1430559,MR1462060}, the predictor-corrector step due to Mehrotra \cite{mehrotra_implementation_1992}, and the Lanczos algorithm for computing step lengths \cite{toh_note_2002}. 

When translating polynomial constraints into semidefinite constraints (see Section~\ref{sec:MPMP}), one obtains for each polynomial constraint a number of semidefinite constraints which use the same positive semidefinite matrix variables. By using sampling  it is possible to keep the rank of the constraint matrices low \cite{lofberg_coefficients_2004}. Together, this leads to a clustered low-rank semidefinite program, with clusters of constraints using the same positive semidefinite variables, and low-rank constraint matrices. We assume these clusters are connected only through free scalar variables.

We therefore consider semidefinite programs with $J$ clusters of the form
\begin{maxi}
  {}{\sum_{j=1}^J\langle C^{j}, Y^{j}\rangle  + \langle c, y\rangle }{}{}
  \label{pr:sdp}
  \addConstraint{\big\langle A_*^{j}, Y^{j}\big\rangle + B^jy}{= b^j,\quad}{ j=1,\ldots,J}
  \addConstraint{Y^{j}}{\succeq 0,}{j=1,\ldots,J,}
\end{maxi}
where we optimize over the vector of free variables $y$ and the positive semidefinite block matrices $Y^j = \mathrm{diag}(Y^{j,1}, \ldots, Y^{j,L_j})$. Here $\langle c, y \rangle$ is the Euclidean inner product, and we use the notation
\[
\langle A_*^{j}, Y^{j}\rangle = \big( \langle A_t^{j}, Y^{j}\rangle \big)_{t \in T_j},
\]
where $\langle A_t^{j}, Y^{j}\rangle$ is the trace inner product. 

The semidefinite program is defined by the symmetric matrices $C^j$ and $A_t^j$, the matrices $B^j$, and the vectors $c \in \R^N$ and $b^j \in \R^{T_j}$. We assume the matrix $A_t^j$ is of the form
\begin{equation}
\label{eq:constraintmatrices}
A_t^j = \bigoplus_{l=1}^{L_j} \sum_{r,s=1}^{R_j(l)} A_t^{j,l}(r, s) \otimes E_{r,s}^{R_j(l)},
\end{equation}
with $A_t^{j,l}(r, s)$ a matrix of low rank and $A_t^{j,l}(r, s)^{\sf T} =  A_t^{j,l}(s, r)$. Here $E_{r,s}^n$ is the $n \times n$ matrix with a one at position $(r,s)$ and zeros otherwise. 

Internally, we represent the  blocks $A_t^{j,l}(r,s)$ in the form
\begin{equation}\label{eq:nonsymmetric}
\sum_i \lambda_i v_i w_i^{\sf T},
\end{equation}
where we do not require the rank $1$ terms to be symmetric (even if the block $A_t^{j,l}(r,s)$ itself is symmetric). Allowing for nonsymmetric matrices in the rank $1$ decomposition is more general than what is done in \cite{simmons-duffin_semidefinite_2015,anjos_implementation_2012,benson_dsdp5_nodate}, and as explained in Section~\ref{sec:kissing_number} this can be important for performance.

We interpret \eqref{pr:sdp} as the dual of the semidefinite program
\begin{mini}
  {}{\sum_{j=1}^J \langle b^j, x^j \rangle}{}{}
  \label{pr:primal_sdp}
  \addConstraint{\sum_{j=1}^J (B^j)^{\sf T} x^j}{ = c}{}
  \addConstraint{X^{j}= \sum_{t \in T_j} x_t^j A_t^{j} - C^{j}}{\succeq 0,\quad}{j=1,\ldots,J,}
\end{mini}
where we optimize over the vectors of free variables $x^j$ and the positive semidefinite block matrices $X^j = \mathrm{diag}(X^{j,1}, \ldots, X^{j,L_j})$. 

Using the notation $X$ for the block matrix $\mathrm{diag}(X^1,\ldots,X^J)$ and $Y$ for the block matrix $\mathrm{diag}(Y^1,\ldots,Y^J)$, the duality gap for primal feasible $(x,X)$ and dual feasible $(y,Y)$ is given by
\[
b^{\sf T}x - \langle C,Y\rangle - c^{\sf T} y = \langle X, Y \rangle.
\]
We assume strong duality holds, so that if $(x,X)$ and $(y,Y)$ are optimal, then $\langle X, Y \rangle = 0$, and hence $XY = 0$.

The primal-dual algorithm starts with infeasible solutions $(x,X)$ and $(y,Y)$, where $X$ and $Y$ are positive definite. At each iteration, a Newton direction $(dx,dX,dy,dY)$ is computed for the system of primal and dual linear constraints and the centering condition $XY = \beta \mu I$. Here $\mu$ is the surrogate duality gap $\langle X, Y \rangle$ divided by the size of the matrices, and $\beta$ is a solver parameter between $0$ and $1$. Then $(x,X,y,Y)$ is replaced by $(x+s dx, X+sdX, y+sdy, Y+sdY)$ for some step size $s$ that ensures the matrices stay positive definite. Here we only discuss the process of finding the search direction since exploiting the special form \eqref{eq:constraintmatrices} happens in this part of the algorithm. See \cite{simmons-duffin_semidefinite_2015} for more details on the remaining parts of the algorithm.

To compute the Newton search direction we replace the variables $(x,X,y,Y)$ by $(x+dx,X+dX,y+dy,Y+dY)$ in the primal and dual constraints, which gives
\begin{align}
    X^{j} + dX^{j} &= \sum_{t\in T_j} (x_t^j+dx_t^j) A_t^{j} - C^{j},\label{eq:dX}\\
    \sum_{j=1}^J (B^j)^{\sf T}(x^j+dx^j) &= c,\label{eq:dx}\\
    \big\langle A_*^{j},Y^{j}+dY^{j}\big\rangle+ B^j(y+dy) &= b^j\label{eq:dy}.
\end{align}
Then we apply the same substitution in the centering condition and linearize to get 
\begin{equation}
X^jY^j + X^j dY^j +dX^j Y^j = \mu I.\label{eq:cent}
\end{equation}
Substituting the expression for $dX^j$ from \eqref{eq:dX} into \eqref{eq:cent} and then the expression for $dY^j$ from \eqref{eq:cent} into \eqref{eq:dy} gives
\begin{align*}
&\Big\langle A_*^j, Y^j + (X^j)^{-1} \Big( \mu I - X^jY^j - \Big(\sum_{t \in T_j}(x_t^j+dx_t^j)A_t^j - C^j - X^j\Big)Y^j \Big)\Big\rangle\\
&\quad + B^j(y+dy) = b^j.
\end{align*}
Together with constraint \eqref{eq:dx} (which is responsible for the last row in the system) this can be written as the following linear system in $dx$ and $dy$:
\begin{align*}
    \begin{bmatrix}
        S^1 & 0 & \cdots & 0 & -B^1\\
        0 & S^2 & \cdots & 0 & -B^2\\
        \vdots & \vdots & \ddots & \vdots & \vdots \\
        0 & 0 & \hdots & S^J & -B^J \\
        (B^1)^{\sf T} & (B^2)^{\sf T} & \hdots & (B^J)^{\sf T} & 0
    \end{bmatrix}
    \begin{bmatrix}
    dx^1 \\ dx^2 \\ \vdots \\
    dx^J \\ dy 
    \end{bmatrix} 
    = 
    \begin{bmatrix}
    - b^1 - \langle A_*^{1}, Z^{1} - Y^1\rangle + B^1 y\\
    - b^2 - \langle A_*^{2}, Z^{2} - Y^2\rangle + B^2 y
    \\ \vdots \\
    - b^J - \langle A_*^{J}, Z^{J} - Y^J\rangle + B^J y\\
    c - \sum_{j=1}^J (B^j)^{\sf T} x^j
    \end{bmatrix},
\end{align*}
Here $Z^j = (X^j)^{-1}((\sum_t x_t^j A_t^j-C^j)Y^j-\mu I)$ and the blocks $S^j$ that form the Schur complement matrix $S = \mathrm{diag}(S^1, \dots, S^J)$ have entries 
\[
S^j_{ab} = \big\langle A_a^{j},  (X^{j})^{-1} A_b^{j} Y^{j}\big\rangle.
\]

The above system can be solved to obtain $dx$ and $dy$. From this $dX$ and $dY$ can be computed, where instead of computing $dY$ as $X^{-1}(\mu I - XY - dX Y)$ we set
\[
dY = \frac{X^{-1}(\mu I - XY - dX Y) + (X^{-1}(\mu I - XY - dX Y))^{\sf T}}{2} 
\]
so that $Y$ stays symmetric.

In general, the computation of the Schur complement matrix $S$ and solving the above linear system are the main computational steps.

Due to the clusters, the matrix $S$ is block-diagonal, so that the Cholesky factorization $S = LL^{\sf T}$ can be computed blockwise. By using the decomposition
\begin{equation*}
    \begin{bmatrix}
    S & -B \\
    B^{\sf T} & 0
    \end{bmatrix} = 
    \begin{bmatrix}
    L & 0 \\
    B^{\sf T}L^{\sf -T} & I
    \end{bmatrix}
    \begin{bmatrix}
    I & 0 \\
    0 & B^{\sf T} L^{\sf -T}L^{-1} B
    \end{bmatrix}
    \begin{bmatrix}
    L^{\sf T} & -L^{-1}B\\
    0 & I
    \end{bmatrix},
\end{equation*}
we can solve the system by solving several triangular systems. The inner matrix $B^{\sf T}L^{\sf -T} L^{-1} B$ is positive definite, so we can again use a Cholesky decomposition. 

Due to the low-rank constraint matrices, we can compute the blocks $S^j$ more efficiently. Suppose for simplicity the constraint matrices are of the form
\[
A^{j,l}_t = \sum_{r=1}^{\eta_t^{j,l}} \lambda_{t,r}^{j,l} v_{t,r}^{j,l} (w_{t,r}^{j,l})^{\sf T},
\]
where $\eta_t^{j,l}$ is the rank of the matrix $A^{j,l}_t$. 
Then we can write
\begin{align*}
    S^j_{ab} &= \sum_{l=1}^{L_j}\langle A_a^{j,l}, (X^{j,l})^{-1} A_b^{j,l} Y^{j,l} \rangle \\
    &= \sum_{l=1}^{L_j} \sum_{r_1=1}^{\eta_a^{j,l}}\sum_{r_2=1}^{\eta_b^{j,l}} \lambda_{a,r_1}^{j,l} \lambda_{b,r_2}^{j,l} \left\langle v_{a,r_1}^{j,l}(w_{a,r_1}^{j,l})^{\sf T}, (X^{j,l})^{-1}v_{b,r_2}^{j,l}(w_{b,r_2}^{j,l})^{\sf T} Y^{j,l}\right\rangle \\
    &= \sum_{l=1}^{L_j} \sum_{r_1=1}^{\eta_a^{j,l}}\sum_{r_2=1}^{\eta_b^{j,l}} \lambda_{a,r_1}^{j,l} \lambda_{b,r_2}^{j,l} \Big((w_{a,r_1}^{j,l})^{\sf T} (X^{j,l})^{-1}v_{b,r_2}^{j,l}\Big)\Big((w_{b,r_2}^{j,l})^{\sf T} Y^{j,l}v_{a,r_1}^{j,l}\Big),
\end{align*}
which shows we can compute $S_{ab}^j$ efficiently by precomputing the  bilinear pairings 
\[
(w_{a,r_1}^{j,l})^{\sf T} (X^{j,l})^{-1}v_{b,r_2}^{j,l} \quad \text{and} \quad (w_{b,r_2}^{j,l})^{\sf T} Y^{j,l}v_{a,r_1}^{j,l}.
\]

In the implementation, we use similar techniques for the more general constraint matrices of the form  \eqref{eq:constraintmatrices}. Since we use high-precision arithmetic, it is beneficial to use matrix-matrix multiplication with subcubic complexity, and therefore we compute the above pairings efficiently by first creating the matrices $V^{j,l}$ and $W^{j,l}$ with the columns $v_{a,r}^{j,l}$ and $w_{a,r}^{j,l}$, respectively, and then performing fast matrix multiplication to compute 
\[
(W^{j,l})^{\sf T}(X^{j,l})^{-1} V^{j,l} \quad \text{and} \quad (W^{j,l})^{\sf T}Y^{j,l} V^{j,l}.
\]

Due to the block structures, the algorithm is relatively easy to parallelize. The best way to parallelize, however, depends on both the problem characteristics and the type of computing system used. The \texttt{SDPB} solver specializes in problems with large amounts of clusters with similar-sized blocks, which can be distributed over different nodes in a multi-node system in which there is communication latency between the nodes \cite{simmons-duffin_semidefinite_2015,landry_scaling_2019}.

Problems in discrete geometry typically consist of few clusters, and have a large variation in both the number of blocks per cluster and in the size of the blocks; see for example Section \ref{sec:kissing_number}. The majority of the workload can be due to a single cluster, hence distributing clusters over nodes in a multi-node system is not a good parallelization strategy in this case. Instead, we focus on distributing the workload over multiple cores in a single-node shared-memory system.

Most of the matrix operations can be done block-wise. We distribute the blocks over the cores such that the workload for each core is about equal. Since the matrices in the products $(W^{j,l})^{\sf T}(X^{j,l})^{-1} V^{j,l}$ and $(W^{j,l})^{\sf T}Y^{j,l} V^{j,l}$ can be very large, we split these multiplications into several parts which we distribute over the cores.

\section{Polynomial matrix programs}
\label{sec:MPMP}

For univariate and multivariate problems, sums-of-squares characterizations, and in particular Putinar's characterization \cite{putinar_positive_1993}, are commonly used to model polynomial constraints as semidefinite constraints. Parrilo and L\"ofberg \cite{lofberg_coefficients_2004} show that by using sampling as opposed to coefficient matching we get a semidefinite programming formulation with low-rank constraint matrices. There has also been research into writing polynomial \emph{matrix} constraints as semidefinite constraints \cite{scherer_matrix_2006,klep_pure_2010}. In this expository section, we consider the combination of multivariate polynomial matrix programs with the sampling approach and show precisely what kind of low-rank constraints appear when reducing these to semidefinite programs.

Let
\[
\mathcal S(G) = \Big\{x \in \R^n : g(x) \geq 0 \text{ for all } g \in G\Big\}
\] 
be the semialgebraic set generated by a finite set of polynomials $G \subseteq \R[x]$ in $n$ variables. Fix $m \in \N$ and define the quadratic module
\[
\mathcal M(G) = \mathrm{cone}\Big\{gp^{\sf T}p : g\in G \cup \{1\},\, p \in \R[x]^{m \times m} \Big\}.
\]
We say $\mathcal M(G)$ is Archimedean if for every $p \in \R[x]^{m \times m}$ there is a $c \in \N$ such that $cI - p^{\sf T}p \in \mathcal M(G)$. As shown in \cite{klep_pure_2010}, this is equivalent to the condition that there is a $c \in \N$ such that $(c-\sum_i x_i^2)I \in \mathcal M(G)$. Intuitively, $\mathcal M(G)$ being Archimedean gives an algebraic certificate for the compactness of $\mathcal S(G)$. A polynomial matrix $f \in \R[x]^{m\times m}$ is said to be positive (semi)definite on $D \subseteq \R^n$ if the matrix $f(x)$ is positive (semi)definite for every $x\in D$. This is denoted by $f \succ 0$ ($f\succeq 0$) on $D$.

A polynomial matrix program is an optimization problem of the form
\begin{maxi}
{}{\langle b, y\rangle}{}{}
\label{polmatrprog}
\addConstraint{P_0^j(x) + \sum_{i=1}^N y_i P_i^j(x)}{\succeq 0 \text { on } \mathcal S(G_j),\quad }{j=1,\ldots,J},
\end{maxi}
where we optimize over the vector $y \in \R^N$. The problem is defined by the sets $G_1,\ldots,G_J$ (where each of these sets may consist of polynomials in a different number of variables), the matrix polynomials $P_i^j$, and the vector $b$. 
In \cite{simmons-duffin_semidefinite_2015} the special case with $n=1$ and $G_j = \{x\}$ is considered.

For the $m=1$ case, positivity constraints on a set $\mathcal S(G)$ are modeled using weighted sums-of-squares polynomials. Such polynomials are trivially nonnegative on $\mathcal S(G)$, and by a theorem of Putinar \cite{putinar_positive_1993} one can prove convergence when increasing the maximum degree of the sum-of-squares polynomials, assuming $\mathcal M(G)$ is Archimedean. A similar approach can be used for polynomial matrix programs. For this, we need a generalization of Putinar's theorem for matrix polynomials by Hol and Scherer \cite{scherer_matrix_2006} (with a different proof given by Klep and Schweighofer in \cite{klep_pure_2010}).

\begin{theorem}[{\cite{scherer_matrix_2006,klep_pure_2010}}]
\label{thm:klep}
Let $f \in \mathbb{R}[x]^{m \times m}$ and  $G\subseteq \mathbb{R}[x]$ finite. Suppose $\mathcal M(G)$ is Archimedean. If $f \succ 0$ on $\mathcal S(G)$, then $f \in \mathcal M(G)$.
\end{theorem}

Similar to the non-matrix case, the requirement $f \succ 0$ can be weakened to $f \succeq 0$ when considering the univariate case where $\mathcal S(G)$ is $\R$, $\R_{\geq a}$, or $[a,b]$. In addition, $\mathcal M(G)$ is not required to be Archimedean in that case.

To state the relaxed problem, we consider the truncated quadratic module generated by $G$:
\[
\mathcal M^d(G) = \mathrm{cone}\Big\{gp^{\sf T}p : g\in G \cup \{1\},\, p \in \mathbb{R}[x]^{m \times m},\, \deg(gp^{\sf T}p) \leq d\Big\}.
\]
This gives 
\begin{maxi}
{}{\langle b, y\rangle}{}{}
\label{pr:MPMP_relax}
\addConstraint{P_0^j + \sum_{i=1}^N y_i P_i^j}{\in \mathcal M^d(G_j),\quad }{j=1,\ldots,J}.
\end{maxi}

Let $p^*$ and $p^*_{d}$ denote the optimal values of problem \eqref{polmatrprog} and \eqref{pr:MPMP_relax}, respectively. For all $d$ we have $p^*_{d}\leq p^*_{d+1} \leq p^*$ and the following corollary whose proof is standard shows convergence under mild conditions.

\begin{corollary}
Suppose \eqref{polmatrprog} is strictly feasible and $\mathcal M(G_j)$ is Archimedean for every $j$.  Then
$p^*_{d} \to p^*$ as $d \to \infty$.
\end{corollary}

It follows from the next lemma that we can model the elements in $\mathcal M^d(G)$ using positive semidefinite matrices. Let $b_d(x)$ be a vector whose elements form a basis for the polynomials of degree at most $d$.
\begin{lemma}
For $f\in \R[x]^{m \times m}$ with $deg(f) = 2d$ we have
$
f = p^{\sf T}p
$
for some $p \in \R[x]^{t \times m}$
if and only if 
\[
f =  (b_d(x) \otimes I_m)^{\sf T}Y (b_d(x) \otimes I_m)
\]
for some positive semidefinite matrix $Y$.
\end{lemma}
\begin{proof}
Let $\delta$ be the length of $b_d(x)$. Then $p = Z (b_d(x) \otimes I_m)$ for some $Z \in \R^{t \times m\delta}$, and hence  $p^{\sf T}p = (b(x) \otimes I_m)^{\sf T} Z^{\sf T}Z (b(x) \otimes I_m)$. Now note that $Y = Z^{\sf T}Z$ is positive semidefinite and any positive semidefinite matrix admits such a decomposition.
\qed\end{proof}

The above lemma shows the elements of $\mathcal M^d(G_j)$ are of the form
\[
    \sum_{g \in G_j \cup \{1\}} g(x) (b_{d-\lfloor \deg(g)/2\rfloor}(x) \otimes I_{m_j})^{\sf T}Y_g^j (b_{d-\lfloor \deg(g)/2\rfloor}(x) \otimes I_{m_j}),
\]
for positive semidefinite matrices $Y_g^j$.
The entry on row $r$ and column $s$ is equal to 
\begin{align*}
&\sum_{g \in G_j \cup \{1\}} g(x) \Big\langle (b_{d-\lfloor \deg(g)/2\rfloor}(x) \otimes I_{m_j})^{\sf T}Y_g^j (b_{d-\lfloor \deg(g)/2\rfloor}(x) \otimes I_{m_j}),E_{r,s}^{m_j} \Big\rangle \\
&\quad= \sum_{g \in G_j \cup \{1\}} g(x) \Big\langle Y_g^j, b_{d-\lfloor \deg(g)/2\rfloor}(x)b_{d-\lfloor \deg(g)/2\rfloor}(x)^{\sf T} \otimes E_{r,s}^{m_j} \Big\rangle
\end{align*}
where $E_{r,s}^{m_j} = e_r e_s^{\sf T}$ is the standard basis of $\R^{m_j \times m_j}$.

This leads to the optimization problem
\begin{maxi*}
  {}{\langle b, y \rangle }{}{}
  \addConstraint{P_0^j(x)_{r,s} + \sum_{i=1}^N y_i P_i^j(x)_{r,s}}{= \big\langle M^{j}(x)_{r,s}, Y^j\big\rangle, \quad}{ j \in [J], \, r,s \in [m_j]}  \addConstraint{Y^j}{\succeq 0,}{j \in [J]},
\end{maxi*}
where
\[
M^j(x)_{r,s} = \bigoplus_{g\in G_j \cup \{1\}} g(x) b_{d-\lfloor \deg(g)/2\rfloor}(x)b_{d-\lfloor \deg(g)/2\rfloor}(x)^{\sf T} \otimes E_{r,s}^{m_j}.
\]

In applications, the formulation \eqref{polmatrprog} can be extended to a more general problem, where the polynomial matrices depend linearly on positive semidefinite matrix variables in addition to the free variables and where there are linear equality constraints on the free variables and the additional positive semidefinite matrix variables. We, therefore, use the following general form for a sums-of-squares problem:
\begin{maxi}
  {}{\sum_{j=1}^J \langle C^j, Y^j\rangle  + \langle b, y \rangle }{}{}
  \label{eq:pol_program}
  \addConstraint{\big\langle A_*^{j}(x), Y^j\big\rangle + B^j(x) y}{= c^j(x),\quad}{ j=1,\ldots,J}
  \addConstraint{Y^j}{\succeq 0,}{j=1,\ldots,J}.
\end{maxi}
Here we use the notation
\[
\langle A_*^{j}(x), Y^{j}\rangle = \big( \langle A_q^{j}(x), Y^{j}\rangle \big)_{q=1,\ldots,Q_j},
\]
where
\[
A_q^j(x) = \bigoplus_{l=1}^{L_j} \sum_{r,s=1}^{R_j(l)} A_q^{j,l}(r,s)(x) \otimes E_{r,s}^{R_j(l)},
\]
and where the matrices $A_q^{j,l}(r,s)$ are of low rank. Here $Q_j$ is the number of polynomial constraints in the $j$-th cluster, where different clusters are only linked via the free variables but not via the positive semidefinite matrix variables. Moreover, $L_j$ specifies the number of blocks on the diagonal of $A_q^j(x)$, where the $l$-th block is a $R_j(l) \times R_j(l)$ block matrix.

In \eqref{eq:pol_program} we optimize over the free variables $y$ and the positive semidefinite block-diagonal matrices $Y^j$. Here, the blocks  $Y^j_l$ do not all have to correspond to sum-of-squares polynomial matrices, and not all constraints have to be polynomial constraints (that is, some constraint can use degree $0$ polynomials). See the examples in Section~\ref{sec:applications}. 

The usual approach for converting a problem of the form \eqref{eq:pol_program} to a semidefinite program is to equate the coefficients of the polynomials in a common basis. This potentially gives sparsity but destroys the low-rank structure of the matrices. Instead, we use the sampling approach as introduced by L\"ofberg and Parrilo in \cite{lofberg_coefficients_2004} and used by Simmons-Duffin in \cite{simmons-duffin_semidefinite_2015}. For each $1 \leq j \leq J$ and $1 \leq q \leq Q_j$ we define a unisolvent set of points $M_q^j$ for the polynomial subspace spanned by the entries of $A_q^j(x)$, the entries in the $q$-th row of $B^j(x)$, and the $q$-th entry of $c^j(x)$. This is a set of points such that any polynomial in this space which is zero on $M_q^j$ is identically zero. Then we consider the linear constraints 
\[
\big\langle A_q^{j}(x'), Y^j\big\rangle + (B^j(x') y)_q = c_q^j(x')
\]
for $x' \in M_q^j$.  That is, when going from \eqref{eq:pol_program} to \eqref{pr:sdp}, we set
\[
 T_j = \{ (q, x') : q=1, \ldots, Q_j, \, x' \in M^j_q\}.
\]



\section{Combining symmetry reduction and sampling}
\label{sec:symmetry}

In this section, we investigate the combination of symmetry reduction with sampling as opposed to coefficient matching. For this, we first give an exposition of the symmetry reduction approach for polynomial optimization by Gatermann and Parrilo \cite{gatermann_symmetry_2004}, where we give more details for the constructions important in this paper. For notational simplicity we consider only polynomial programs as opposed to matrix polynomial programs.

Suppose the polynomials $P_0, \ldots, P_N$ and the set $\mathcal S(G)$ are invariant under the action of a finite group $\Gamma$, and assume the polynomials in $G$ are chosen to be $\Gamma$-invariant (which is in fact always possible if $\mathcal S(G)$ is invariant under $\Gamma$; see Appendix~\ref{appendix:sym}). Then the sums-of-squares characterization for a constraint of the form
\[
P_0 + \sum_{i=1}^N y_i P_i \ge 0 \text{ on } \mathcal S(G)
\]
can be written more efficiently. 

Let $\Gamma$ be a finite group with a linear action on $\C^n$, and define the representation $L \colon \Gamma \to \mathrm{GL}(\C[x])$ by
$
L(\gamma)p(x) = p(\gamma^{-1} x).
$
Here $\mathrm{GL}(\C[x])$ is the automorphism group of the vector space $\C[x]$. Let $\widehat \Gamma$ be a complete set of irredicible representations $(\pi, V_\pi)$ of $\Gamma$. By Maschke's theorem, we get the decomposition
\[
\C[x] = \bigoplus_{\pi \in \widehat \Gamma} \bigoplus_{i} H_{\pi,i},
\] 
where $H_{\pi, i}$ is equivalent to $V_\pi$. Since the space of homogeneous polynomials of degree $k$ is invariant under the action of $\Gamma$, we may assume that  $H_{\pi,i}$ is spanned by homogeneous polynomials of the same degree.

Since $\Gamma$ is finite we can choose a basis $e_{\pi,1},\ldots,e_{\pi,d_\pi}$ of $V_\pi$ in which the linear operators $\pi(\gamma)$ are unitary matrices. We want to define bases $e_{\pi,i,1}$, $\ldots$, $e_{\pi,i,d_\pi}$ of $H_{\pi,i}$ which are symmetry adapted in the sense that the restriction of $L(\gamma)$ to the invariant subspace $H_{\pi,i}$  in this basis  is exactly $\pi(\gamma)$. 

Such a basis exists since  $H_{\pi,i}$ is equivalent to $V_\pi$, so there are $\Gamma$-equivariant isomorphisms $T_{\pi,i} \colon V_\pi \to H_{\pi,i}$ and we can define  $e_{\pi,i,j} = T_{\pi,i} e_{\pi,j}$. Then it follows that $L(\gamma) e_{\pi,i,j} = \sum_k \pi(\gamma)_{k,j} e_{\pi,i,k}$. As described in \cite[Section 2.7]{serre_linear_1996} a symmetry-adapted basis can be constructed by defining the operators
\begin{equation}\label{eq:symmetrize}
p_{j,j'}^\pi = \frac{d_\pi}{|\Gamma|} \sum_{\gamma \in \Gamma} \pi(\gamma^{-1})_{j',j}L(\gamma),
\end{equation}
and then choosing bases $\{e_{\pi,i,1}\}_i$ of $\mathrm{Im}(p_{1,1}^\pi)$ and setting $e_{\pi,i,j} = p_{j,1}^\pi e_{\pi,i,1}$. Then,
\begin{align*}
    L(\tilde\gamma)e_{\pi,i,j} &= \frac{d_\pi}{|\Gamma|} \sum_{\gamma\in\Gamma} \pi(\gamma^{-1})_{1,j} L(\tilde\gamma\gamma)e_{\pi,i,1} = \frac{d_\pi}{|\Gamma|} \sum_{\gamma\in\Gamma} \pi(\gamma^{-1} \tilde\gamma)_{1,j} L(\gamma)e_{\pi,i,1}\\
    &=\frac{d_\pi}{|\Gamma|} \sum_{\gamma\in\Gamma}\sum_{k=1}^{d_\pi} \pi(\gamma^{-1})_{1,k} \pi(\tilde\gamma)_{k,j} L(\gamma) e_{\pi,i,1} = \sum_{k=1}^{d_\pi} \pi(\tilde\gamma)_{k,j} e_{\pi,i,k}.
\end{align*}

If the irreducible representations occurring in the decomposition of $\C[x]$ are of real type, then we can choose bases so that the matrices $\pi(\gamma)$ are orthogonal. If moreover $\R[x]$ is $\Gamma$-invariant, then the above explicit construction shows we can take the symmetry adapted basis to be real: Since \eqref{eq:symmetrize} is a real operator, we can choose a real basis $\{e_{\pi,i,1}\}_i$ of the image of $p_{1,1}^\pi$ and $e_{\pi,i,j} = p_{j,1}^\pi e_{\pi,i,1}$ will be real as well.

For each $\pi$ we define the matrix polynomial $E_\pi$ by
\[
E_\pi(x)_{i,i'} = \frac{1}{d_\pi} \sum_{j=1}^{d_\pi} e_{\pi,i,j}(x) \overline{e_{\pi,i',j}(x)}
\]
That the matrices $E_\pi(x)$ are $\Gamma$-invariant follows from the alternative definition
\[
E_\pi(x)_{i,i'} = \frac{1}{|\Gamma|} \sum_{\gamma \in \Gamma} e_{\pi,i,j}(\gamma^{-1} x) \overline{e_{\pi,i',j}(\gamma^{-1} x)}
\]
(for any $1 \leq j \leq d_\pi$), which follows from
\begin{align*}
&\frac{1}{|\Gamma|} \sum_{\gamma \in \Gamma} e_{\pi,i,j}(\gamma^{-1} x) \overline{e_{\pi,i',j}(\gamma^{-1} x)} \\
&\quad= \frac{1}{|\Gamma|} \sum_{\gamma \in \Gamma} \sum_{l=1}^{d_\pi} \pi(\gamma^{-1})_{l,j} e_{\pi,i,l}(x) \sum_{k=1}^{d_\pi} \overline{\pi(\gamma^{-1})_{k,j} e_{\pi,i',k}(x)} \\
&\quad= \sum_{l,k=1}^{d_\pi} e_{\pi,i,l}(x) \overline{e_{\pi,i',k}(x)} \frac{1}{|\Gamma|}\sum_{\gamma \in \Gamma} \pi(\gamma^{-1})_{l,j} \overline{\pi(\gamma^{-1})_{k,j}}  \\
&\quad= \sum_{l,k=1}^{d_\pi} e_{\pi,i,l}(x) \overline{e_{\pi,i',k}(x)} \,\frac{\delta_{kl}}{d_\pi},
\end{align*}
where we use the Schur orthogonality relations (see, e.g., \cite[Section 2.2]{serre_linear_1996}) in the last equality.  

For $d\in \N$, we define $E_\pi^d(x)$ as the submatrix of $E_\pi(x)$ indexed by rows and columns for which $\deg e_{\pi,i,j}(x) \leq d$. The following proposition shows how these matrices can be used to parametrize Hermitian sum-of-squares polynomials  by Hermitian positive semidefinite matrices. If the symmetry-adapted basis is real, then the matrices $E_\pi^d$ are symmetric and we can parametrize real sum-of-squares polynomials by positive semidefinite matrices.

\begin{proposition}
\label{prop:G-inv_SOS}
If $p \in \C[x]_{\leq 2d}$ is a $\Gamma$-invariant Hermitian sum-of-squares polynomial, then there are Hermitian positive semidefinite matrices $C_\pi$ such that
\[
p(x) = \sum_\pi \Big\langle C_\pi, E^d_\pi(x) \Big\rangle.
\]
\end{proposition}
\begin{proof}
Define the column vector $b(x) = (e_{\pi,i,j}(x))_{\pi,i,j}$ with $\deg e_{\pi,i,j}(x) \leq d$. Since $p$ is a Hermitian sum-of-squares polynomial it can be written as 
\[
p(x) = \sum_i (c_i^* b(x))^* (c_i^* b(x)),
\]
so there exists a Hermitian positive semidefinite matrix $A$ such that
\[
p(x) = b(x)^* A b(x).
\]
Let $\rho(\gamma)$ be the matrix obtained by expressing the restriction of $L(\gamma)$ to $\C[x]_{\leq 2d}$ in the symmetry adapted basis, so that 
\begin{equation}\label{eq:block}
\rho(\gamma) = \bigoplus_\pi I_{m_\pi} \otimes \pi(\gamma)
\end{equation}
and $\rho(\gamma) b(x) = b(\gamma^{-1} x)$ for all $x$ and $\gamma$. Define 
\[
B = \frac{1}{|\Gamma|}\sum_{\gamma \in \Gamma} \rho(\gamma)^* A \rho(\gamma).
\]
Then $\rho(\gamma)^* B \rho(\gamma) = B$ for all $\gamma\in \Gamma$ and
\begin{align*}
p(x) &= \frac{1}{|\Gamma|}\sum_{\gamma \in \Gamma} b(\gamma^{-1} x)^* A b(\gamma^{-1}x) =  b(x)^* B  b(x).
\end{align*}
By writing $B$ in the block form $B = ( B_{(\pi,i),(\pi',i')} )$ and using \eqref{eq:block} we get
\[
\pi(\gamma) B_{(\pi,i),(\pi',i')} =  B_{(\pi,i),(\pi',i')} \pi'(\gamma)
\]
for all $\gamma\in \Gamma$. By Schur's lemma $B_{(\pi,i),(\pi',i')}$ is a multiple of the identity if $\pi = \pi'$ and zero otherwise (see, e.g., \cite[Section 2.2]{serre_linear_1996}). This shows $B = \bigoplus_\pi \frac{1}{d_\pi}C_\pi \otimes I_{d_\pi}$ for positive semidefinite matrices $C_\pi$. We then have
\begin{align*}
p(x) &= \langle B, b(x)b(x)^*\rangle\\
&= \left\langle \bigoplus_\pi \frac{1}{d_\pi}C_\pi \otimes I_{d_\pi}, b(x)b(x)^*\right\rangle \\
&= \sum_\pi \sum_{j=1}^{d_\pi} \sum_{i,i'=1}^{m_\pi} \frac{1}{d_\pi}(C_\pi)_{i,i'} e_{\pi,i,j}(x) e_{\pi,i',j}(x)^* \\
&= \sum_{\pi}  \big\langle C_\pi, E^d_\pi(x) \big\rangle,
\end{align*}
which completes the proof.
\qed\end{proof}
In applications, the symmetry groups often are reflection groups (see, e.g., \cite{MR3679945} and Section~\ref{sec:kissing_number}), and as described in \cite{gatermann_symmetry_2004} for these groups we can choose the symmetry adapted basis in such a way that the matrices $E_\pi(x)$ have a tensor structure. By, e.g., \cite[Section 3.6]{humphreys_reflection_1990}, $\C[x]$ is a free module over the invariant ring $\C[x]^\Gamma$ of rank $|\Gamma|$. Moreover, the span of any $\C[x]^\Gamma$ module basis of $\C[x]$ is equivalent to the regular representation of $\Gamma$. This means we have the decomposition
\[
\C[x] = \C[x]^\Gamma \bigoplus_\pi \bigoplus_{i=1}^{d_\pi} V_{\pi,i},
\]
where $V_{\pi,i} \subseteq \C[x]$ is equivalent to $V_\pi$. As before, we may assume that $V_{\pi,i}$ is spanned by homogeneous polynomials of the same degree. 

Let $\{f_{\pi,i,j}\}$ be a symmetry adapted basis of
\[
\bigoplus_\pi \bigoplus_{i=1}^{d_\pi} V_{\pi,i}.
\]
Then we can choose the symmetry-adapted basis of $\C[x]$ to be of the form
\[
e_{\pi,(i,k),j}(x) = f_{\pi, i, j}(x)\, w_k(x),
\]
where $w_k(x)$ is a basis of $\C[x]^\Gamma$.
The matrix $E_\pi(x)$ can therefore be written as 
\[
E_{\pi}(x)_{(i,k),(i',k')} = \sum_{j=1}^{d_\pi} f_{\pi,i,j}(x) w_k(x)\overline{f_{\pi,i',j}(x)w_{k'}(x)},
\]
i.e.,
\[
E_\pi(x) = \Pi_\pi(x) \otimes w(x)w(x)^*,
\]
where $\Pi_\pi(x)$ is the matrix given by $\Pi_\pi(x)_{i,i'} = \sum_{j=1}^{d_\pi} f_{\pi,i,j}(x) \overline{f_{\pi,i',j}(x)}$. 

\medskip

From now on we assume that we can and do choose the symmetry-adapted basis  to be real so that the matrices $E_\pi(x)$ are symmetric. We then replace the constraint
\[
P_0 + \sum_{i=1}^N y_i P_i \geq 0 \text{ on } \mathcal S(G)
\]
by the condition that there are positive semidefinite matrices $Y_{g,\pi}$ for which
\begin{equation}\label{eq:harmonic}
P_0(x) + \sum_{i=1}^N y_i P_i(x) = \sum_{g \in G \cup \{1\}} g(x) \sum_\pi \left\langle Y_{g,\pi}, E_\pi^{d-\lfloor\deg(g)/2\rfloor}(x) \right\rangle.
\end{equation}

As in Section~\ref{sec:MPMP} we want to model constraint \eqref{eq:harmonic} by the linear constraints obtained from evaluating it at a unisolvent set of points. Since $P_i$, $g$, and $E_\pi$ are all $\Gamma$-invariant, it is sufficient to consider a unisolvent set for the subspace of $\Gamma$-invariant polynomials of degree at most $2d$.
A unisolvent set is said to be minimal if any strict subset is not unisolvent. In the following proposition we show that the constraints arising from evaluating \eqref{eq:harmonic} at a minimal unisolvent set are linearly independent, which is essential for the solver.

\begin{proposition}\label{prop:minimalunisolvent}
Evaluating \eqref{eq:harmonic} at a minimal unisolvent set $M$ for the subspace of $\Gamma$-invariant polynomials of degree at most $2d$  yields $|M|$ linearly independent constraints in the variables of the optimization problem.
\end{proposition}
\begin{proof}
By the proof of Proposition~\ref{prop:G-inv_SOS}, we can express any polynomial as a linear combination of the entries of the matrix $\oplus_\pi E_\pi^d$, which means that these entries span the space of $\Gamma$-invariant polynomials of degree at most $2d$.

In particular, there is a subset $\{b_j(x)\}_j$ of the entries of $\oplus_\pi E_\pi^d$ which forms a basis for the invariant polynomials. Since the unisolvent set $M$ is minimal, the vectors $(b_j(x'))_j$ are linearly independent for $x' \in M$. This implies the matrices $\oplus_\pi E_\pi^d(x')$ are linearly independent, and hence the linear combinations
\[
\sum_\pi \Big\langle Y_{1,\pi}, E_\pi^d(x')\Big\rangle,
\]
for $x' \in M$, are linearly independent. This shows that evaluating \eqref{eq:harmonic} on $M$ yields $|M|$ linearly independent constraints.
\qed\end{proof}

A symmetric polynomial optimization problem often arises as the symmetrization of a non-symmetric problem. For this reason, we want to know when and how a minimal unisolvent set for the non-symmetric problem gives a minimal unisolvent set for the symmetrized problem. The following lemma gives a sufficient condition for this to be the case.

\begin{lemma}
\label{lem:R_minimal_unisolvent}
Suppose $M$ is a $\Gamma$-invariant, minimal unisolvent set for $\R[x]_{\leq 2d}$. Then any set of representatives $R$ for the orbits in $M$ is minimal unisolvent for $\R[x]^\Gamma_{\leq 2d}$.
\end{lemma}
\begin{proof}
Suppose $p \in \R[x]^\Gamma_{\leq 2d}$ satisfies $p(r) = 0$ for every $r\in R$. Then for every $x \in M$ there are $r \in R$ and $\gamma\in\Gamma$ with $x = \gamma r$, so $p(x) = p(\gamma r) = p(r) = 0$. Hence $p = 0$ by unisolvence of $M$. So $R$ is unisolvent for $\R[x]^\Gamma_{\leq 2d}$. 

Consider the projection operator $\mathcal P \colon \R[x]_{\le 2d} \to \R[x]_{\le 2d}$  defined by
\[
\mathcal P p(x) = \frac{1}{|\Gamma|} \sum_{\gamma\in\Gamma} p(\gamma^{-1} x).
\]
Let $b(x)$ be a column vector where the first entries form a basis for the eigenspace with eigenvalue $1$ of $\mathcal P$ and the remaining entries form a basis for the kernel of $\mathcal P$. Consider the Vandermonde matrix $V = (b(x))_{x \in M}$, where $M$ indexes the columns. By minimal unisolvence of $M$, the matrix $V$ is square and nonsingular. We can perform invertible column operations to transform the first $|R|$ columns into $|\Gamma r|^{-1} \sum_{x \in \Gamma r} b(x)$, where $\Gamma r$ is the orbit represented by $r \in R$. Now note that the first $\dim \R[x]^\Gamma_{\leq 2d}$ entries in each column stay the same and the other entries in the first $|R|$ columns become $0$. That is, after performing these column operations the matrix is of the form
\[
 V' = \begin{bmatrix}
 V_{11} & V_{12} \\
 0 & V_{22}
 \end{bmatrix}.
\]
Note that $V_{11}$ must be square: $R$ is unisolvent, so the number of columns $|R|$ is at least the number of rows $\dim \R[x]^\Gamma_{\leq 2d}$, and $V'$ is nonsingular. Hence $|R| = \dim \R[x]^\Gamma_{\leq 2d}$, i.e., $R$ is minimal unisolvent.
\qed\end{proof}

There are group actions for which a $\Gamma$-invariant, minimal unisolvent set $M$ does not exist. Since $\Gamma$ is a finite group, there are a finite number of orbit sizes $o_i$. Suppose $M$ and $R$ are as in Lemma~\ref{lem:R_minimal_unisolvent}. Then $M$ can be decomposed into orbits such that there are $k_i$ orbits with orbit size $o_i$. This directly means that $\sum_i k_i o_i = |M| = \dim \R[x]_{\leq 2d}$ because $M$ is minimal unisolvent. Furthermore, the minimial unisolvence of $R$ implies that $\sum_i k_i = |R| = \dim \R[x]_{\leq 2d}^\Gamma$. When this system of equations does not have a nonnegative integer solution, there does not exist a $\Gamma$-invariant, minimal unisolvent set. Moreover, even if the system does have a nonnegative integer solution, it is possible that there does not exist an invariant, minimal unisolvent set $M$. Such an example can be found in Appendix~\ref{appendix:min_unisolvent}.

In the following lemma we show that for a group action permuting $n+1$ affinely independent vectors (which includes the important case of permuting coordinates), an invariant, minimal unisolvent set exists. For this we use a geometric criterion of minimal unisolvence by Chung and Yao \cite{chung_lattices_1977}: a set $M$ of size $\dim \R[x]_{\leq d}$ is minimal unisolvent for $\R[x]_{\le d}$ if   
for every $x \in M$ there are hyperplanes $H_{x,1},\ldots,H_{x,d}$ such that
\[
M \cap \left(\bigcup_{l=1}^d H_{x,l}\right) = M \setminus \{x\}.
\]

\begin{lemma}
Let $S_{n+1}$ be the symmetric group on $n+1$ elements. Suppose $\Gamma \subseteq S_{n+1}$ acts on $\R^n$ by permuting $n+1$ affinely independent vectors $v_1,\ldots,v_{n+1}$. Then there is a $\Gamma$-invariant, minimal unisolvent set $M$ for $\R[x]_{\leq d}$.
\end{lemma}
\begin{proof}
Without loss of generality we can take $\Gamma = S_{n+1}$. For each $x\in\R^n$ there are unique coefficients $\alpha_1^x,\ldots,\alpha_{n+1}^x$ such that $\sum_{i} \alpha_i^x = 1$ and $x = \sum_i \alpha_i^x v_i$. Let $M$ be the set of points $x$ for which 
\[
\alpha_1^x,\ldots,\alpha_{n+1}^x \in \left\{ \frac{k}{d} : k=0,\ldots,d\right\}.
\]
Note that $M$ is invariant under the action of $\Gamma$. 

For each $x \in M$ we define the hyperplanes
\[
H_{x,(j,k)} = \big\{y \in \R^n : d\alpha_j^y = k\big\}
\]
for $1 \leq j \leq n+1$ and $0 \leq k < d\alpha_j^x$. Here $H_{x,(j,k)}$ is a hyperplane because $H_{x,(j,k)}-k/dv_j$ is the  affine hull of the vectors $(1-k/d) v_i$ with $i \neq j$. Note that this gives $\sum_{j} d \alpha_j^x = d$ hyperplanes for each $x$. 

Since $k < d \alpha_j^x$ we have $x \not\in H_{x,(j,k)}$. Moreover, for any $y \in M \setminus \{x\}$ there is an $i$ such that $\alpha^y_i < \alpha^x_i$, i.e., $y \in  H_{x, (i,d \alpha_i^y)}$. So the geometric characterization is satisfied, hence $M$ is minimal unisolvent.
\qed\end{proof}

\section{Computing good sample points and bases}
\label{sec:samples}

To model the constraints of the form \eqref{eq:harmonic} using sampling we need to find a minimal unisolvent set for the space of $\Gamma$-invariant polynomials of degree at most $2d$ (where $\Gamma$ is the trivial group if there is no symmetry). In Section~\ref{sec:symmetry} we explain when such a set can be derived from a minimal unisolvent set for $\R[x]_{\leq d}$. For the interior point method, however, we do not just want the set to be minimal unisolvent, but also to have good numerical conditioning. Moreover, as explained in Section~\ref{sec:symmetry}, if the group is a reflection group (which is often the case in practice), then the matrix $E_\pi^d(x)$ in \eqref{eq:harmonic} is a submatrix of 
\[
\Pi_\pi(x) \otimes w(x)w(x)^{\sf T},
\]
where for $w$ we can choose any vector whose entries form a basis for the $\Gamma$-invariant polynomials of degree at most $d$. Note that in the case of no symmetry, $\Pi_\pi(x)$ is just the $1 \times 1$ identity matrix.

In \cite{sommariva_computing_2009} Sommariva and Vianello discuss a greedy method for finding a good sample set and a good basis for quadrature problems. In this section we adapt this to find a good minimal unisolvent set of sample points as well as a good basis for the entries of $w(x)$.

The space of polynomials corresponding to a constraint in the polynomial matrix program \eqref{polmatrprog} is given by
\[
W = \R[x]_{\leq 2d}^{\Gamma_j}.
\]
Let $v$ be a vector whose entries form a  basis of $W$ such that $\deg(v_k)$ is nondecreasing in $k$. Let $M$ be a set containing $m$ distinct points $x_1,\ldots, x_m$ from the semialgebraic set $\mathcal S(G)$, where $m$ is at least $\dim(W)$. Here the idea is that we take $m$ much larger than $\dim(W)$ and later select a subset of good points. The Vandermonde matrix $V$ with respect to $v$ and $M$ is defined by $V_{lk} = v_k(x_l)$, for $l=1,\dots,m$ and $k=1,\dots,\dim(W)$. The set $M$ is unisolvent for $W$ if and only if the kernel of $V$ is trivial, and minimal unisolvent if additionally $V$ is square. 

The Fekete points for a compact domain and a polynomial space are defined as the points in the domain that maximize the determinant of the Vandermonde matrix in absolute value. Because a basis change multiplies the determinant by a constant not depending on the samples, the Fekete points do not depend on the basis. Instead of computing the Fekete points, which is hard to do in general \cite{taylor_algorithm_2000}, Sommariva and Vianello select a subset of a set of candidate points which approximately maximizes the determinant, the approximate Fekete points. This can be done greedily by a QR factorization of $V^{\sf T}$ with column pivoting; the points corresponding to the first $\dim(W)$ pivots approximately maximize the determinant. We now let $M'$ be the subset of $M$ corresponding to the first $\dim(W)$ many pivots, and let $V'$ be the square submatrix of $V$ by selecting the corresponding rows. If the original set $M$ was unisolvent, then the determinant of $V$ will be nonzero.

Because we use sample points to express the sums-of-squares constraints as semidefinite constraints, 
it is desirable for the numerical conditioning that the entries of $w$ are orthogonal with respect to the chosen sample set $M'$. Such a basis can be obtained by a QR factorization of the Vandermonde matrix $V'$, as also mentioned by L\"ofberg and Parrilo \cite{lofberg_coefficients_2004}. Here the columns of $Q$ represent a new basis of $W$, where each basis element defines a polynomial by its evaluations on the sample set $M'$. For the polynomials $w_i$ we now choose the polynomials in this basis which have degree at most $d$. 

Note that for the implementation it is not necessary to recover the coefficients of the polynomials $w_i$ in the monomial basis since we only ever use the evaluations of \eqref{eq:harmonic} on the points in $M'$. This means we can directly use the columns of $Q$ (which are in fact the coefficients with respect to the Lagrange basis for the set $M'$). In fact, we extend the computer algebra system \texttt{AbstractAlgebra.jl} \cite{Nemo.jl-2017} we use in the interface to our solver with a new type called \texttt{SampledMPolyElem}, which is a polynomial defined only by its evaluations on a given set. This is very useful for modeling for instance the right-hand side of \eqref{eq:harmonic}, where we have a mixture of polynomials (the weights $g(x)$ and the entries of the matrices $\Pi_\pi(x)$) and sampled polynomials (the entries of $w(x)$).

\medskip
\centerline { $ * \quad * \quad * $ }
\medskip

Since the matrices to which the linear algebra routines described above need to be applied are much larger than the matrices considered in the solver, we want to perform the operations in machine precision as much as possible (otherwise the preprocessing may become more expensive than solving the semidefinite program). Here we have to be careful that after performing the QR factorization of $V'$, the degrees of the polynomials corresponding to the columns in $Q$ are still nondecreasing in the column index, because we need the first columns to correspond to a basis of the polynomials in $W$ of degree at most $d$. We, therefore, have to adapt \cite[Algorithm 2]{sommariva_computing_2009} to our setting; this adaptation is also implemented in the package \texttt{ClusteredLowRankSolver.jl}.

First, we improve the basis by computing the QR factorization of $V$ in machine precision. We then compute the matrix $VR^{-1}$ using high-precision arithmetic and replace $V$ with it. By using high-precision arithmetic  we ensure that the degrees of the polynomials corresponding to the columns of $V$ will still be nonincreasing in the column index. Although the columns of the new $V$ will not be orthogonal (due to numerical issues in the QR factorization), they will be more orthogonal than before. We can optionally repeat this process a few times. 

We then compute a pivoted QR factorization of $V^{\sf T}$ in machine precision, and as described above use it to select a suitable set of samples and define $V'$ by selecting the corresponding rows of $V$.

To find a basis that is orthogonal with respect to this sample set we compute the QR factorization of $V'$ in machine precision and compute $V'R^{-1}$ in high-precision arithmetic. Since the numerical conditioning of $V'$ should now be relatively good, the columns of $V'R^{-1}$  will indeed be near orthogonal. We then select appropriate columns to use as basis polynomials for the entries of $w$, where as described above we do not need to compute the coefficients in the monomial basis.

\section{Applications}
\label{sec:applications}

In this section we consider two applications from discrete geometry: the kissing number problem and the binary sphere packing problem. The first application showcases the speed of the solution approach for a symmetric multivariate polynomial optimization problem. The second application showcases the stability of the sampling approach for a univariate polynomial matrix problem.

\subsection{Kissing number problem}\label{sec:kissing_number} 

A subset $C$ of the unit sphere $S^{n-1} = \{v \in \R^n : \langle v, v\rangle = 1\}$ is a spherical $\theta$-code if $\langle v, v' \rangle \leq \cos(\theta)$ for all distinct $v,v' \in C$. In discrete geometry we are interested in the maximum size $A(n, \theta)$ of such a set. For $\cos\theta=1/2$ this is the kissing number problem, where we ask for the maximum number of unit spheres that can simultaneously touch a central unit sphere. The kissing number problem has a rich history; see \cite{pfender_ziegler_kissing_2004} for background information.
 
In \cite{bachoc_new_2007}, Bachoc and Vallentin introduce a three-point semidefinite programming bound that gives many of the best-known upper bounds on $A(n,\theta)$. Here we give the formulation of this bound leaving out an ad-hoc $2\times 2$ matrix which does not contribute numerically \cite{dostert_exact_2021}. 

Define the matrices $Y_k^n(u,v,t)$ by 
\[
Y_k^n(u,v,t)_{ij} = u^i v^j (1-u^2)^{k/2} (1-v^2)^{k/2} P_k^{n-1}\mathopen{}\left(\frac{t-uv}{\sqrt{(1-u^2)(1-v^2)}}\right),
\]
where $P_k^n$ is the $k$-th degree Gegenbauer polynomial with parameter $n/2-1$, normalized such that $P_k^n(1) = 1$. Define the matrices $\bar Y_k^n(u,v,t)$ by
\[
\bar Y_k^n(u,v,t) = \frac{1}{6}\sum_{\sigma\in S_3}\sigma Y_k^n(u, v, t),
\]
where $\sigma \in S_3$ acts on $Y_k^n$ by permuting its arguments. With these matrices the three-point bound is the problem
\begin{mini*}
  {}{1+\sum_{k=0}^{2d} a_k + \big\langle\bar Y_0^n(1,1,1), F_0 \big\rangle}{}{}
  \addConstraint{\sum_{k=0}^{2d} a_k P_k^n(u) + 3\sum_{k=0}^d  \big\langle \bar Y_k^n(u,u,1), F_k \big\rangle}{\leq -1,\quad}{u \in [-1, \cos\theta]}
    \addConstraint{\sum_{k=0}^d  \big\langle \bar Y_k^n(u,v,t), F_k \big\rangle}{\leq 0,\quad}{ (u,v,t) \in \mathcal S(G)}
    \addConstraint{a_0,\ldots, a_{2d} \geq 0, F_0, \ldots, F_d \succeq 0,}{}{}
\end{mini*}
where 
\[
    \mathcal S(G) = \big\{(u,v,t) : -1 \leq u,v,t \leq \cos\theta, \,1+2uvt -u^2-v^2-t^2 \geq 0\big\}.
\]
Note that the multivariate constraints are symmetric under the action of the symmetric group $S_3$ on three elements, so that we use the techniques from Section~\ref{sec:symmetry} and \ref{sec:samples}.

We can view the above problem as being an extension of \eqref{polmatrprog} with the additional positive semidefinite matrix variables $a_k$ of size $1$ and $F_k$ of size $d-k+1$. Using the approach from Section~\ref{sec:MPMP} we then obtain a problem in the form \eqref{eq:pol_program} and after sampling a semidefinite program of the form \eqref{pr:sdp}. We consider two ways of doing this. 

In the first approach we view the problem as an extension of \eqref{polmatrprog} without free variables, but where the polynomial
constraints depend linearly on $a_k$ and $F_k$. After converting to \eqref{eq:pol_program}, the semidefinite program has positive semidefinite matrix variables $a_k$, $F_k$, and variables corresponding to the sum-of-square polynomials. When going from \eqref{eq:pol_program} to \eqref{pr:sdp} we have to evaluate the polynomials $P_k^n$, the polynomial matrices $\bar Y_k^n$, and the matrices arising from the sum-of-squares polynomials at the samples. Here it is beneficial to factor $\bar Y_k^n$ symbolically before evaluating it at the samples. The matrices $\bar Y_k^n(u,v,t)$ have rank at most $3$ after evaluation at a point $(u,v,t)$. It is however not clear whether we can write $\bar Y_k^n(u,v,t)$ as a sum of three symmetric rank $1$ matrix polynomials  of the form $\lambda(x)v(x)v(x)^{\sf T}$. We, therefore, decompose $\bar Y_k^n(u,v,t)$ as a sum of nonsymmetric rank $1$ matrix polynomials and we use the fact that our semidefinite programming solver supports nonsymmetric rank $1$ factors in the symmetric constraint matrices. 
For large $d$ this will be the fastest approach. 

For intermediate $d$ we consider the alternative approach where we use free variables for the entries of $a_k$ and $F_k$, write the polynomial constraints in terms of these free variables, and use additional rank $1$ linear constraints to link the free variables to newly introduced positive semidefinite matrices $a_k'$ and $F_k'$. This approach can be faster for small and intermediate $d$ and uses less memory, which can be important for practical computations. For the computations discussed below (where $d$ is at most $20$) we use this approach. For more complicated problems, with for instance polynomial inequality constraints on unions of basic semialgebraic sets, this approach may be very useful since it allows more fine-grained clustering of the positive semidefinite matrices which the solver can exploit.

Initial computations for the three-point bound were performed by Bachoc and Vallentin using \texttt{CSDP} \cite{borchers_CSDP_1999}, but since this is a machine precision solver it was not possible to go beyond $d=10$ \cite{bachoc_new_2007}. Mittelman and Vallentin \cite{mittelmann_high_2010} then used the high precision solvers \texttt{SDPA-QD} and \texttt{SDPA-GMP} to perform computations up to $d=14$. Later Machado and Oliviera \cite{machado_improving_2018} applied symmetry reduction and used \texttt{SDPA-GMP} to compute bounds up to degree $d=16$. 

For $d=16$ and using the same symmetry reduction, our solver gives a speedup by a factor of $28$ over the approach using \texttt{SDPA-GMP} using $8$ cores, and a speedup by a factor of $9.6$ using $1$ core. Here we use the same hardware and the default settings of \texttt{SDPA-GMP}, except that we use $256$ bits of floating point precision for all three-point  bound computations. 

It would be interesting to compare this to timings one would obtain using direct optimization over sum-of-squares polynomials, as developed by Skajaa, Ye, Papp, and Y{\i}ld{\i}z  \cite{skajaa_homogeneous_2015,papp_homogeneous_2018,papp_sum--squares_2018}. However, for this their approach would first have to be extended to semidefinite programs with polynomial constraints, and a high-precision solver would have to be implemented.

As can be seen in Figure~\ref{fig:time_threepointbound_linear} and Figure~\ref{fig:time_threepointbound_log} the factor by which our solver is faster than the approach with \texttt{SDPA-GMP} increases with $d$. For the approach using \texttt{SDPA-GMP} the computation time theoretically scales as $d^{12}$ when sparsity is not exploited; in practice we see that it scales as $d^{10.1}$. With our approach the computation time theoretically scales as $d^9$, and in practice we observe a scaling of $d^{8.55}$. The discrepancy between theory and practice for our approach can in part be explained by the fact that matrix multiplication in Arb is faster than cubic \cite{johansson_faster_2019}.

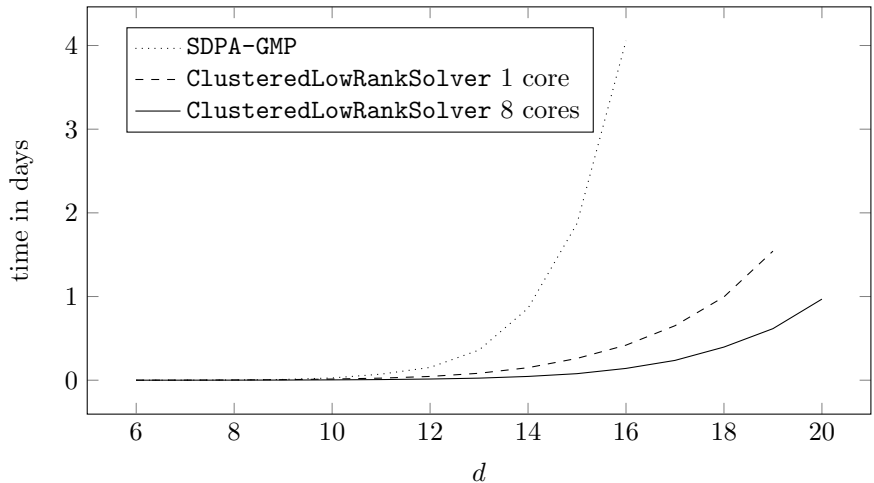
\begin{figure}
    \centering
    \begin{tikzpicture}
\begin{axis}[ylabel={time in days},ylabel style = {yshift={-1em}}, xlabel={$d$}, xmin={5.0}, xmax={21.0}, legend style={at={(0.05,0.95)
}, anchor={north west}},legend cell align={left}, width={120mm}, height={70mm}, xtick={{10,20}}, xticklabels={{$10$,$20$}}, extra x ticks={{6,8,12,14,16,18}}, extra x tick labels={6,8,12,14,16,18}]
    \addplot[dotted]
        table[row sep={\\}]
        {
            \\
            6  0.00023084404369212963  \\
            7  0.0008855452125231482  \\
            8  0.0029611971285416665  \\
            9  0.009010452356226852  \\
            10  0.026820384692349535  \\
            11  0.07239924046471065  \\
            12  0.15263824499689815  \\
            13  0.360471167773125  \\
            14  0.8602249955590046  \\
            15  1.8730173381529167  \\
            16  4.056067032465197  \\
        }
        ;
    \addlegendentry {\texttt{SDPA-GMP}}
    \addplot[dashed]
        table[row sep={\\}]
        {
            \\
            6       0.0004371859022864589   \\
            7       0.0010735797688916878   \\
            8       0.0029377950341613205   \\
            9       0.00646187830026503     \\
            10      0.013819564293932031    \\
            11      0.02478117893415469     \\
            12      0.04505661885495539     \\
            13      0.0829454838981231      \\
            14      0.149576953958582       \\
            15      0.2604028888074336      \\
            16      0.41879238627299115     \\
            17      0.6511371186127265      \\
            18      0.995839359536767       \\
            19      1.5424868345122646      \\
        }
        ;
    \addlegendentry {\texttt{ClusteredLowRankSolver} 1 core}
    \addplot[solid]
        table[row sep={\\}]
        {
            \\
            6  0.00017826331986321343  \\
            7  0.00038261787207038316  \\
            8  0.0008991980442294368  \\
            9  0.001838558934353016  \\
            10  0.003921085474667726  \\
            11  0.007574711447512662  \\
            12  0.014030572465724415  \\
            13  0.02476262925951569  \\
            14  0.04511357049699183  \\
            15  0.07818823021871073  \\
            16  0.14097209284702936  \\
            17  0.23672282923426893  \\
            18  0.39615767190853757  \\
            19  0.616100584803908  \\
            20  0.9690950763556693  \\
        }
        ;
    \addlegendentry {\texttt{ClusteredLowRankSolver} 8 cores}
    
\end{axis}
\end{tikzpicture}
    \vspace{-2em}
    \caption{The time needed to compute the three-point bound for the kissing number in dimension $n=4$ for several degrees $d$ on a linear scale, using \texttt{SDPA-GMP} and \texttt{ClusteredLowRankSolver}.}
    \label{fig:time_threepointbound_linear} 
\end{figure}
\begin{figure}
    \centering
    \begin{tikzpicture}
\begin{loglogaxis}[ylabel={time in seconds},ylabel style = {yshift={-0.5em}}, xlabel={$d$}, xmin={5.0}, xmax={21.0}, legend style={at={(0.05,0.95)
}, anchor={north west}},legend cell align={left}, width={120mm}, height={70mm}, xtick={{10,20}}, xticklabels={{$10$,$20$}}, extra x ticks={{6,8,12,14,16,18}}, extra x tick labels={6,8,12,14,16,18}]
    \addplot[dotted]
        table[row sep={\\}]
        {
            \\
            6  19.944925375  \\
            7  76.511106362  \\
            8  255.847431906  \\
            9  778.503083578  \\
            10  2317.281237419  \\
            11  6255.294376151  \\
            12  13187.944367732  \\
            13  31144.708895598  \\
            14  74323.439616298  \\
            15  161828.698016412  \\
            16  350444.191604993  \\
        }
        ;
    \addlegendentry {\texttt{SDPA-GMP}}
    \addplot[dashed]
        table[row sep={\\}]
        {
            \\
            6       37.77286195755005       \\
            7       92.75729203224182       \\
            8       253.8254909515381       \\
            9       558.3062851428986       \\
            10      1194.0103549957275      \\
            11      2141.093859910965       \\
            12      3892.8918690681458      \\
            13      7166.489808797836       \\
            14      12923.448822021484      \\
            15      22498.809592962265      \\
            16      36183.662173986435      \\
            17      56258.24704813957       \\
            18      86040.52066397667       \\
            19      133270.86250185966      \\
        }
        ;
    \addlegendentry {\texttt{ClusteredLowRankSolver} 1 core}
    \addplot[solid]
        table[row sep={\\}]
        {
            \\
            6  15.40195083618164  \\
            7  33.0581841468811  \\
            8  77.69071102142334  \\
            9  158.8514919281006  \\
            10  338.7817850112915  \\
            11  654.455069065094  \\
            12  1212.2414610385895  \\
            13  2139.4911680221558  \\
            14  3897.812490940094  \\
            15  6755.463090896606  \\
            16  12179.988821983337  \\
            17  20452.852445840836  \\
            18  34228.022852897644  \\
            19  53231.09052705765  \\
            20  83729.81459712982  \\
        }
        ;
    \addlegendentry {\texttt{ClusteredLowRankSolver} 8 cores}
\end{loglogaxis}
\end{tikzpicture}
    \vspace{-2em}
    \caption{The time needed to compute the three-point bound for the kissing number in dimension $n=4$ for several degrees $d$ on a log-log scale, using \texttt{SDPA-GMP} and \texttt{ClusteredLowRankSolver}.}
    \label{fig:time_threepointbound_log}
\end{figure}
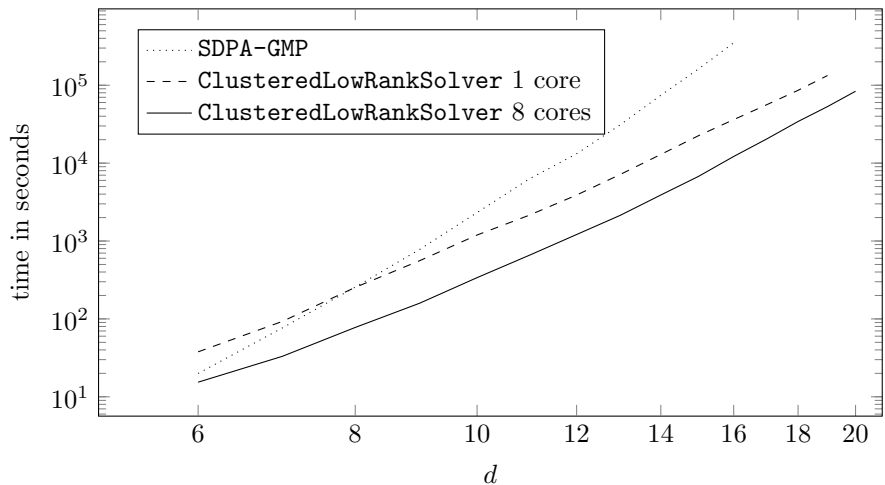

Because of the speedup of our approach we can perform computations up to $d=20$ within a reasonable time frame (extrapolating Figure~\ref{fig:time_threepointbound_log} shows that the approach using \texttt{SDPA-GMP} would have been slower by a factor $40$ for $d=20$). In Table~\ref{tab:kissing_number_bounds} we show the kissing number bounds for $d =16,\ldots,20$ for dimensions up to $24$. Dimension $2$, $8$, and $24$ are omitted since the linear programming bound is sharp in these dimensions. After rounding down to the nearest integer, this improves the best known upper bounds in dimensions $11$ through $23$. 

Rigorous verification of these bounds can be done using standard interval-arithmetic techniques (see, e.g., \cite{de_laat_upper_2014,machado_improving_2018}), the only caveat being that we first need to apply the reverse basis transformation obtained in Section~\ref{sec:samples} to the obtained solution. We did not perform this verification procedure since our main goal here is to showcase the speed of our approach for this type of problems. Note that the bounds we report for $d=16$ are slightly different from the bounds reported in \cite{machado_improving_2018} since their verification procedure increases the bounds by a configurable parameter $\varepsilon > 0$.

\begin{table}
    \centering
   {\footnotesize \begin{tabular}{cccc|cccc}
\toprule
$n$ & lower bound & $d$ & upper bound & $n$ & lower bound & $d$ & upper bound \\
\midrule
 3 & 12 & 16 & 12.368580 & 14 & 1606 & 16 & 3177.7812 \\
  &  & 17 & 12.364503 &  &  & 17 & 3176.4354 \\
  &  & 18 & 12.360782 &  &  & 18 & 3175.3519 \\
  &  & 19 & 12.357869 &  &  & 19 & 3174.7746 \\
  &  & 20 & 12.353979 &  &  & 20 & \underline{3174}.1890 \\
 \midrule
 4 & 24 & 16 & 24.056877 & 15 & 2564 & 16 & 4858.1937 \\
  &  & 17 & 24.053495 &  &  & 17 & 4856.4186 \\
  &  & 18 & 24.051431 &  &  & 18 & 4855.1064 \\
  &  & 19 & 24.048769 &  &  & 19 & 4854.3872 \\
  &  & 20 & 24.047205 &  &  & 20 & \underline{4853}.7561 \\
  \midrule
 5 & 40 & 16 & 44.981014 & 16 & 4320 & 16 & 7332.7695 \\
  &  & 17 & 44.976437 &  &  & 17 & 7329.8545 \\
  &  & 18 & 44.973846 &  &  & 18 & 7325.5713 \\
  &  & 19 & 44.971353 &  &  & 19 & 7322.5461 \\
  &  & 20 & 44.970252 &  &  & 20 & \underline{7320}.1068 \\
  \midrule
 6 & 72 & 16 & 78.187644 & 17 & 5346 & 16 & 11014.169 \\
  &  & 17 & 78.173268 &  &  & 17 & 11004.299 \\
  &  & 18 & 78.163358 &  &  & 18 & 10994.873 \\
  &  & 19 & 78.151981 &  &  & 19 & 10984.895 \\
  &  & 20 & 78.143569 &  &  & 20 & \underline{10978}.622 \\
  \midrule
 7 & 126 & 16 & 134.26988 & 18 & 7398 & 16 & 16469.091 \\
  &  & 17 & 134.21522 &  &  & 17 & 16445.457 \\
  &  & 18 & 134.17305 &  &  & 18 & 16431.764 \\
  &  & 19 & 134.13115 &  &  & 19 & 16418.296 \\
  &  & 20 & 134.10709 &  &  & 20 & \underline{16406}.358 \\
  \midrule
 9 & 306 & 16 & 363.67296 & 19 & 10668 & 16 & 24575.872 \\
  &  & 17 & 363.59590 &  &  & 17 & 24516.534 \\
  &  & 18 & 363.50742 &  &  & 18 & 24463.542 \\
  &  & 19 & 363.41738 &  &  & 19 & 24443.476 \\
  &  & 20 & 363.34567 &  &  & 20 & \underline{24417}.472 \\
  \midrule
 10 & 500 & 16 & 553.82278 & 20 & 17400 & 16 & 36402.676 \\
  &  & 17 & 553.57125 &  &  & 17 & 36296.753 \\
  &  & 18 & 553.38179 &  &  & 18 & 36250.908 \\
  &  & 19 & 553.21188 &  &  & 19 & 36218.806 \\
  &  & 20 & 553.05527 &  &  & 20 & \underline{36195}.348 \\
  \midrule
 11 & 582 & 16 & 869.23401 & 21 & 27720 & 16 & 53878.723 \\
  &  & 17 & 868.82650 &  &  & 17 & 53724.682 \\
  &  & 18 & 868.45366 &  &  & 18 & 53647.201 \\
  &  & 19 & 868.15131 &  &  & 19 & 53567.621 \\
  &  & 20 & \underline{868}.01070 &  &  & 20 & \underline{53524}.085 \\
  \midrule
 12 & 840 & 16 & 1356.5778 & 22 & 49896 & 16 & 81376.460 \\
  &  & 17 & 1356.1536 &  &  & 17 & 81085.186 \\
  &  & 18 & 1355.8837 &  &  & 18 & 80962.164 \\
  &  & 19 & 1355.4776 &  &  & 19 & 80860.092 \\
  &  & 20 & \underline{1355}.2976 &  &  & 20 & \underline{80810}.158 \\
  \midrule
 13 & 1154 & 16 & 2066.3465 & 23 & 93150 & 16 & 123328.40 \\
  &  & 17 & 2065.5348 &  &  & 17 & 122796.10 \\
  &  & 18 & 2064.9493 &  &  & 18 & 122657.49 \\
  &  & 19 & 2064.4859 &  &  & 19 & 122481.07 \\
  &  & 20 & \underline{2064}.0029 &  &  & 20 & \underline{122351}.67 \\
\bottomrule
\end{tabular}
}
   \vspace{1em}
    \caption{Three-point bounds for the kissing number problem in dimensions $3$-$23$. Dimension $8$ is omitted since there the linear programming bound is sharp. New records after rounding down to the nearest integer are underlined. Lower bounds are taken from \cite{machado_improving_2018}.}
    \label{tab:kissing_number_bounds}
\end{table}

\subsection{Binary sphere packing}
\label{sec:binary}

The $m$-sphere packing problem asks for the optimal sphere packing density in Euclidean space using spheres of $m$ prescribed sizes. For $m=1$ this is the well-known sphere packing problem, for which the linear programming bound by Cohn and Elkies has been used to prove the optimality of the $E_8$ root lattice in $\R^8$ and the Leech lattice in $\R^{24}$ \cite{cohn_new_2003,viazovska_sphere_2017,cohn_sphere_2017}. In \cite{de_laat_upper_2014}, de Laat, Oliveira, and Vallentin generalize this bound to the $m$-sphere packing problem, and they use this to compute bounds for the binary sphere packing problem in dimensions $2,\ldots,5$ with radii $(r/1000,1)$, $r = 200,\ldots,1000$. For smaller $r$ and dimensions higher than $5$ computations were prevented by numerical instabilities. Here we show this bound can be modeled as a univariate polynomial matrix program using matrices of size $m$. We use this to perform computations for a larger range of radii and higher dimensions, which allows us to make new qualitative observations about the behavior of these bounds.

Before stating the bound, we recall some definitions. A function $f \colon \R^n \to \R$ is a Schwarz function if it is infinitely differentiable, and if any derivative of $f(v)$ multiplied with any monomial in $v_1,\ldots,v_n$ is a bounded function. If $f \colon \R^n \to \R$ is a radial function, then for $t \ge 0$ we write $f(t)$ for the common value of $f$ on vectors $v$ of norm $t$. A matrix-valued Schwartz function $f \colon \R^n \to \R^{m \times m}$ is a matrix-valued functions whose every component function is a Schwartz function. We define the Fourier transform of such functions entrywise:
\[
\widehat f(v)_{rs} = \int f(w)_{rs} e^{-2\pi i  \langle v, w \rangle} \, dw.
\]

\begin{theorem}[\protect{\cite[Theorem 5.1]{de_laat_upper_2014}}]
\label{thm:sphere_packing}
Let $R_1,\ldots,R_m>0$. Suppose $f \colon \mathbb{R}^n \to \mathbb{R}^{m \times m}$ is a radial, matrix-valued Schwartz function that satisfies the following:
\begin{enumerate}
    \item The matrix $\widehat{f}(0)- W$ is positive semidefinite, where 
    \[
    W_{rs} = (\mathrm{vol}\, B(R_r))^{1/2}(\mathrm{vol}\, B(R_s))^{1/2}
    \]
    and $B(R)$ is the ball of radius $R$. 
    \item The matrix $\widehat{f}(t)$ is positive semidefinite for every $t>0$.
    \item  $f_{rs}(t)\leq 0$ whenever $t \geq R_r+R_s$, for $r,s = 1,\ldots,m$.
\end{enumerate}
Then the density of any sphere packing of spheres of radii $R_1,\ldots,R_m$ in the Euclidean space $\mathbb{R}^n$ is at most $\max \{f_{rr}(0): r=1,\ldots,m\}.$
\end{theorem}

Following \cite{de_laat_upper_2014} we parametrize $\widehat f$ as
\[
\widehat f(t) = \sum_{k=0}^d A^{(k)} t^{2k} e^{-\pi t^2},
\]
for some symmetric matrices $A^{(0)},\ldots,A^{(d)}$. By \cite[Lemma 5.2]{de_laat_upper_2014} we then have
\[
f(t) = \sum_{k=0}^d A^{(k)} \frac{k!}{\pi^k}L_k^{n/2-1}(\pi t^2)e^{-\pi t^2},
\]
where $L_k^{n/2-1}$ is the degree $k$ Laguerre polynomial with parameter $n/2-1$. Because positive factors do not influence the positivity, the Gaussians can be ignored, resulting in the following polynomial matrix program (where some of the constraints use only degree $0$ polynomials):
\begin{mini*}
    {}{M}{}{}
    \addConstraint{-W + A^{(0)}}{ \succeq 0}{}
    \addConstraint{\sum_{k=0}^d A^{(k)} x^{k} }{\succeq 0 \text{ on } \R_+}{}
    \addConstraint{- \sum_{k=0}^d A^{(k)}_{rs} \frac{k!}{\pi^k} L_k^{n/2-1}(\pi x)}{ \geq 0 \text{ on } S(G_{rs}),}{ \quad \text{ for } 1\leq r \leq s \leq m}
    \addConstraint{M - \sum_{k=0}^{d} A^{(k)}_{rr} \frac{k!}{\pi^k} L_k^{n/2-1}(0)}{\geq 0,}{\quad\text{ for } 1 \leq r \leq m,}
\end{mini*}
where we used the transformation $x = t^2$, and define $G_{rs} = \{x-(R_r+R_s)^2\}$.

The plots in \cite{de_laat_upper_2014} suggest that the binary sphere packing bounds become very bad as the ratio of the radii $r$ tends to $0$. Using the extra data we collect, we see the bounds are not convex in $r$ and in fact the bounds seem to become very good as $r$ tends to $0$. For dimension $24$, the plot in Figure~\ref{fig:bin_sphere_packing_n=24} seems to suggest that as the ratio of the radii $r$ tends to zero the binary sphere packing bound converges to $\Delta + (1-\Delta)\Delta$, where $\Delta$ denotes the optimal sphere packing density. By \cite{de_laat_optimal_2016,cohn_sphere_2017} this is the optimal limiting binary sphere packing density. As the computations for dimension $23$ suggest (see Figure~\ref{fig:bin_sphere_packing_n=23}), the bound more generally may go to $\delta + \delta (1-\delta)$, where $\delta$ is the optimal value of the Cohn-Elkies linear programming bound.

In dimension $2$, the best known upper bound for the binary sphere packing problem is due to Florian \cite{florian_ausfullung_1960}. In \cite{de_laat_upper_2014}, the binary sphere packing bound was calculated for $r\ge 0.2$, and in this regime, the bound is worse than Florian's bound. We compute the bound for $r \geq 0.035$, which is enough to show the bound improves on Florian's bound for small $r$; see Figure \ref{fig:bin_sphere_packing_n=2}. 

Since the binary sphere packing density only depends on the ratio $r = R_1/R_2$ of the radii of the spheres, we may scale both radii by a constant factor $s > 0$ for the calculations. We observed that using scaled radii instead of $(r,1)$ can lead to better numerical conditioning of the problem and to better bounds. The difference can be especially large for small $r$, and decreases when increasing $r$. For example, in dimension 2 with radii $(1/10, 1)$, the bound for degree $d = 31$ is $1.155$ without scaling and $0.9697$ with scaling factor $s = 1.35$.

\begin{figure}
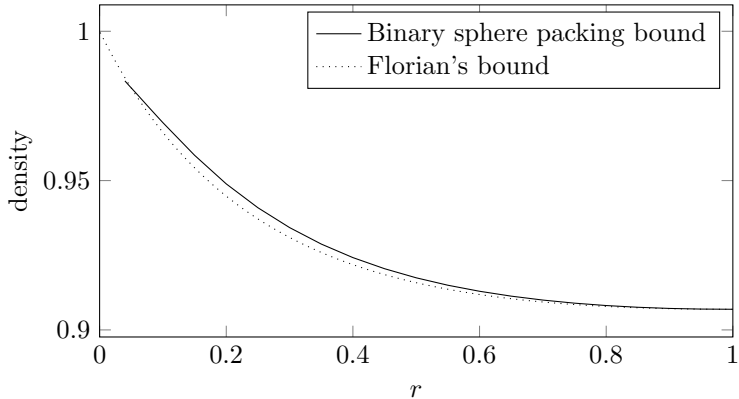

    \centering
    \include{figures/bin_sphere_n=2}
    \vspace{-2em}
    \caption{The binary sphere packing bound in dimension $n=2$ and Florian's bound.}
    \label{fig:bin_sphere_packing_n=2}
\end{figure}

In dimensions $8$ and $24$, we could compute the bound for $r \geq 0.15$ respectively $r \geq 0.3$, see Figures \ref{fig:bin_sphere_packing_n=8} and \ref{fig:bin_sphere_packing_n=24}. This required degrees 71; for small $r$ we computed the bounds with degree $d=91$ to make sure that increasing the degree would not change the plot visibly. We scaled the radii with $s=19/10$. In Figure \ref{fig:bin_sphere_packing_n=24}, we added a dotted line indicating the optimal sphere packing density $\Delta_{24} = \pi^{12}/12!$ \cite{cohn_sphere_2017}, and a dashed line indicating the optimal limiting density. In dimensions $2$ and $8$ we omitted these lines since the curve of the bound is less clear in those dimensions; in dimension 2 the bound is still convex, and in dimension $8$ it is unclear how fast the bound flattens.

\begin{figure}
    \centering
    \begin{tikzpicture}
\begin{axis}[xmin={0.0}, xmax={1.0}, width={10cm}, height={6cm},xlabel ={$r$},ylabel={density},ylabel style = {yshift={-0.5em}}]
    \addplot[black]
        table[row sep={\\}]
        {
            x  y  \\
            0.21  0.39102039430527696  \\
            0.22  0.38799651443839284  \\
            0.23  0.3849136718015576  \\
            0.24  0.38176404307426015  \\
            0.25  0.3785488390115391  \\
            0.26  0.3752668401809887  \\
            0.27  0.3719079433253878  \\
            0.28  0.3684718297031769  \\
            0.29  0.3649599121076039  \\
            0.3  0.36136564117314063  \\
            0.31  0.357684368543368  \\
            0.32  0.3539064165423956  \\
            0.33  0.3500131751412818  \\
            0.34  0.3459776105495515  \\
            0.35  0.34176133364435324  \\
            0.36  0.33730978362332886  \\
            0.37  0.33290112096569263  \\
            0.38  0.32870616959094523  \\
            0.39  0.3247164483170702  \\
            0.4  0.3209206098689275  \\
            0.45  0.3045656706865096  \\
            0.5  0.2918498997756539  \\
            0.55  0.2819489806595425  \\
            0.6  0.2742383370635763  \\
            0.65  0.268252860004393  \\
            0.7  0.26364381256114566  \\
            0.75  0.26014755077984775  \\
            0.8  0.2575627123352037  \\
            0.85  0.25573357607931885  \\
            0.9  0.25453728625199135  \\
            0.95  0.25387611416646644  \\
            1.0  0.2536706910178864  \\
        }
        ;
\end{axis}
\end{tikzpicture}
    \vspace{-2em}
    \caption{Binary sphere packing upper bounds in dimension $n=8$.}
    \label{fig:bin_sphere_packing_n=8}
\end{figure}
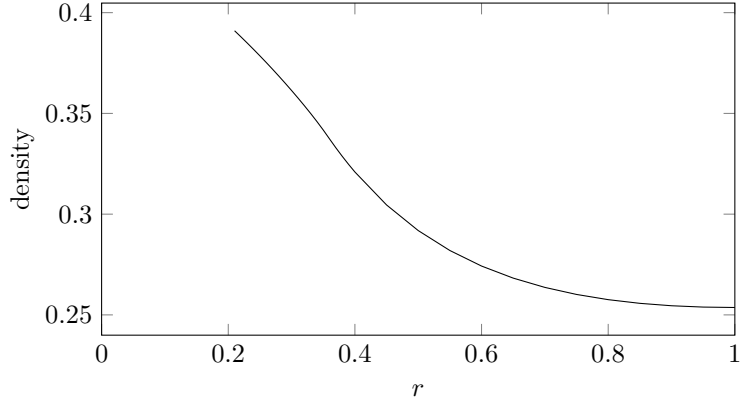

\begin{figure}
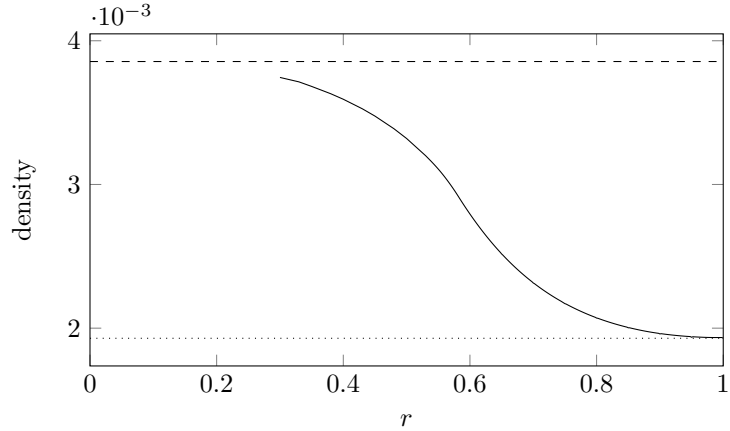

    \centering
    \include{figures/bin_sphere_n=24}
    \vspace{-2em}
    \caption{The binary sphere packing bound in dimension $24$. The dotted line is the maximum density of single sphere packings and the dashed line the optimal density when $r$ tends to $0$.}
    \label{fig:bin_sphere_packing_n=24}
\end{figure}

\begin{figure}
    \centering
    \begin{tikzpicture}
\begin{axis}[xmin={0.0}, xmax={1.0}, width={10cm}, height={6cm},xlabel ={$r$},ylabel={density},ylabel style = {yshift={-1em}}]
    \addplot[black]
        table[row sep={\\}]
        {
            x  y  \\
            0.3  0.0051735730984719475  \\
            0.31  0.005159104267784212  \\
            0.32  0.005140363199237237  \\
            0.33  0.005120748543375096  \\
            0.34  0.0050998900819721775  \\
            0.35  0.005078234056868072  \\
            0.36  0.005055111319346886  \\
            0.37  0.005030706785415098  \\
            0.38  0.005004942307955353  \\
            0.39  0.00497761386688834  \\
            0.4  0.004948703808948707  \\
            0.41  0.004918171718994326  \\
            0.42  0.0048859186598346005  \\
            0.43  0.00485178311610289  \\
            0.44  0.004815690868965111  \\
            0.45  0.004777578462588478  \\
            0.46  0.004737347983348993  \\
            0.47  0.004694768512846661  \\
            0.48  0.00464967565120558  \\
            0.49  0.0046018447830677095  \\
            0.5  0.004551043515527327  \\
            0.53  0.004381756484598824  \\
            0.55  0.004244380571978182  \\
            0.58  0.003983489052123834  \\
            0.6  0.0038087471559178844  \\
            0.63  0.003580729447091603  \\
            0.65  0.0034500974640333886  \\
            0.68  0.0032828334919343884  \\
            0.7  0.003188267141362189  \\
            0.73  0.003068452799169412  \\
            0.75  0.0030014018353830677  \\
            0.78  0.002917336210293085  \\
            0.8  0.0028710167201489362  \\
            0.83  0.0028139763229120293  \\
            0.85  0.002783083131685066  \\
            0.88  0.0027465519666045133  \\
            0.9  0.002728190489916205  \\
            0.93  0.0027078820640775905  \\
            0.95  0.0026989494122374716  \\
            0.98  0.0026922778731953336  \\
            1.0  0.002691994269693558  \\
        }
        ;
    \addplot[black, dashed]
        table[row sep={\\}]
        {
            \\
            0.0  0.0053767417062390525  \\
            1.0  0.0053767417062390525  \\
        }
        ;
    \addplot[black, dotted]
        table[row sep={\\}]
        {
            \\
            0.0  0.002691994269693558  \\
            1.0  0.002691994269693558  \\
        }
        ;
\end{axis}
\end{tikzpicture}
    \vspace{-2em}
    \caption{The binary sphere packing bound in 23 dimensions. The horizontal axis represents the ratio between the radii of the small and the large sphere, and the vertical axis the bound.}
    \label{fig:bin_sphere_packing_n=23}
\end{figure}
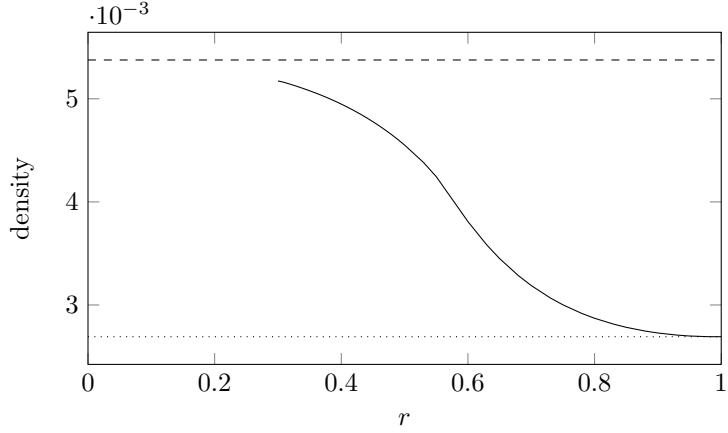

\begin{appendix}

\section{Symmetric description of a symmetric semialgebraic set} \label{appendix:sym}

In this appendix we generalize the argument from \cite[Lemma 3.1]{machado_improving_2018} to arbitrary finite groups, to show that an invariant semialgebraic set is the semialgebraic set of invariant polynomials.

Let $\Gamma$ be a finite group with a linear action on $\R^n$ and consider the linear action on $\R[x]$ by $\gamma p(x) = p(\gamma^{-1} x)$. Let $G$ be a finite set of polynomials such that the semialgebraic set 
$
\mathcal S(G) = \big\{x \in \R^n : g(x) \geq 0 \text{ for } g\in G\big\}
$
is $\Gamma$-invariant. We will show $\mathcal S(G)$ is the semialgebraic set of finitely many invariant polynomials.
Since $\mathcal S(G) = \bigcap_{g\in G} \mathcal S(\Gamma g)$, it is sufficient to show this for  $\mathcal S(\Gamma g)$.

For $g \in G$ we define the $\Gamma$-invariant polynomials
\begin{equation*}
    g_k(x) = \sum_{\substack{\Gamma' \subseteq \Gamma \\ |\Gamma'| = k}} \prod_{\gamma \in \Gamma'} \gamma g(x)
\end{equation*}
for $k=1,\ldots, |\Gamma|$. Then, 
\[
\mathcal S(\Gamma g) \subseteq \mathcal S(\{g_1,\ldots,g_{|\Gamma|}\})
\]
since $\gamma g(x) \geq 0$ for all $\gamma \in \Gamma$ implies $g_k(x) \geq 0$ for all $k$. 

For the reverse inclusion we suppose we have a point $x$ with $g(x) < 0$ and $g_k(x) \geq 0$ for $k \geq 2$. We will argue  that $g_1(x) < 0$. 

Define
\begin{equation*}
T_k = \sum_{\substack{\Gamma' \subseteq \Gamma \\ |\Gamma'| = k \\ e \not\in \Gamma'}} \prod_{\gamma \in \Gamma'} \gamma g(x),
\end{equation*}
where we denote the identity element of $\Gamma$ by $e$.
Then $T_k \leq 0$ for $k=|\Gamma|$, since there are no subsets of $\Gamma$ of size $|\Gamma|$ not containing the identity. If $T_k \leq 0$ for $k \geq 2$, then $g_k(x) = g(x)T_{k-1} + T_k \geq 0$ implies $g(x) T_{k-1} \ge 0$ and hence $T_{k-1} \leq 0$. By induction we then have $T_1 \leq 0$.
Hence,
\[
g_1(x) = \sum_{\gamma \in \Gamma} \gamma g(x) = g(x) + T_1 \leq g(x) < 0.
\]
So $\mathcal S(\Gamma g) = \mathcal S(\{g_1,\ldots,g_{|\Gamma|}\})$.

\section{Nonexistence of an invariant minimal unisolvent set} \label{appendix:min_unisolvent}

In this appendix we give an example in which an invariant minimal unisolvent set does not exist even though this is not apparent from the dimensions of the spaces and the orbit sizes.

Let $\Gamma = D_4$ be the dihedral group $\langle s, d : d^4=s^2=e, \, ds=sd^{-1}\rangle$ with the action on $\R^2$ defined by
\[
\theta(s)(x,y) = (y,x), \qquad
\theta(d)(x,y) = (-y,x).
\]
The invariant ring $\R[x,y]^\Gamma$ is generated by $\phi_1 = x^2+y^2$ and $\phi_2 = x^2y^2$.
Take $2d = 6$. Then $\dim \R[x,y]_{\leq 2d}^\Gamma = 6$, and $\dim \R[x,y]_{\leq 2d} = {2+6 \choose 2} = 28$.

Let $M$ be an invariant set of points with $|M| = \dim \R[x,y]_{\leq 2d} = 28$, and suppose $M$ is unisolvent. We will show that this leads to a contradiction.
By Lemma \ref{lem:R_minimal_unisolvent}, the set $R$ of representatives of the orbits is minimal unisolvent for $\R[x,y]^\Gamma_{\leq 2d}$, and thus has size $\dim \R[x,y]_{\leq 2d}^\Gamma = 6$.

The orbits of $\Gamma$ acting on $\R^2$ have size 1 (the origin $(0,0)$), size $4$ (generated by points $(x,y)$ with $|x| = |y|$, $x = 0$, or $y=0$) and size $8$ (generated by points $(x,y)$ with $|x| \neq |y|$ and $x,y \neq 0$). Since there is only one orbit of odd size, and $|M| = \dim \R[x,y]_{\leq 2d}$ is even, all orbits of $M$ have even size.  From the equations 
\[
4k + 8l = |S| = 28 \quad \text{and} \quad  k + l = |R| = 6
\]
we know that there are $k=5$ points in $R$ corresponding to orbits of size $4$, and $l = 1$ point corresponding to an orbit of size $8$. 

Let $r$ be the norm of a point corresponding to the orbit of size $8$. Note that $\|(x,y)\|^2 = \phi_1(x,y)$ is invariant under the action of $\Gamma$, hence all points in an orbit have the same norm. Define the polynomial 
\[
p(x,y) = xy(x+y)(x-y)(x^2+y^2-r^2).
\]
Then $p(x,y) = 0$ for all $(x,y) \in M$, and $\deg p = 6 = 2d$ so $p\in \R[x,y]_{\leq 2d}$, but $p \neq 0$. Hence $M$ is not unisolvent.

\end{appendix}

\section*{Acknowledgements}

We thank Henry Cohn and Fernando Oliveira for their helpful comments. We also thank the anonymous reviewers whose suggestions helped improve the paper.

\bibliographystyle{amsinit}      
\providecommand{\bysame}{\leavevmode\hbox to3em{\hrulefill}\thinspace}
\providecommand{\MR}{\relax\ifhmode\unskip\space\fi MR }
\providecommand{\MRhref}[2]{%
  \href{http://www.ams.org/mathscinet-getitem?mr=#1}{#2}
}
\providecommand{\href}[2]{#2}


\end{document}